\documentclass[12pt,reqno]{amsart}
\usepackage[left=3cm,top=2cm,right=3cm,bottom=2cm]{geometry}
\usepackage{amssymb}
\usepackage{amsbsy}
\usepackage{epsfig}
\usepackage{tikz}
\usepackage[english]{babel}    
\usepackage{enumitem}
\usepackage{multicol}
\usepackage{blkarray}
\usepackage{tikz}
\usepackage{mathtools}
\usepackage{amsthm}
\usepackage{float}
\usepackage{tikz-network}
\usepackage{hyperref}

\usetikzlibrary{decorations.pathreplacing}

\usepackage{nicematrix}

\usepackage{algorithm}
\usepackage{algpseudocode}

\usepackage{graphicx,subcaption} 

\newtheorem{teore}{Theorem}[section] 
\newtheorem{theorem}[teore]{Theorem} 

\newtheorem{defn}[teore]{Definition}
\newtheorem{lemat}[teore]{Lemma}

\newtheorem{corollary}[teore]{Corollary}

\newtheorem{prop}[teore]{Proposition}
\newtheorem{remark}[teore]{Remark}

\theoremstyle{definition}
\newtheorem{example}[teore]{Example}

\DeclareMathOperator*{\Spec}{Spec}
\DeclareMathOperator*{\DSpec}{DSpec}
\DeclareMathOperator*{\diam}{diam}
\DeclareMathOperator*{\tr}{tr}

\begin{document}
\DeclarePairedDelimiter\ceil{\lceil}{\rceil}
\DeclarePairedDelimiter\floor{\lfloor}{\rfloor}

\title{Diminimal families of arbitrary diameter}
\author[L. E. Allem]{L. Emilio Allem}
\address{UFRGS - Universidade Federal do Rio Grande do Sul, 
Instituto de Matem\'atica, Porto Alegre, Brazil}\email{emilio.allem@ufrgs.br}
\author[R. O. Braga]{Rodrigo O. Braga}
\address{UFRGS - Universidade Federal do Rio Grande do Sul, 
Instituto de Matem\'atica, Porto Alegre, Brazil}\email{rbraga@ufrgs.br}
\author[C. Hoppen]{Carlos Hoppen}
\address{UFRGS - Universidade Federal do Rio Grande do Sul, 
Instituto de Matem\'atica, Porto Alegre, Brazil}\email{choppen@ufrgs.br}
\author[E. R. Oliveira]{Elismar R. Oliveira}
\address{UFRGS - Universidade Federal do Rio Grande do Sul, 
Instituto de Matem\'atica, Porto Alegre, Brazil}\email{elismar.oliveira@ufrgs.br}
\author[L. S. Sibemberg]{Lucas Siviero Sibemberg}
\address{UFRGS - Universidade Federal do Rio Grande do Sul, 
Instituto de Matem\'atica, Porto Alegre, Brazil}\email{lucas.siviero@ufrgs.br}
\author[V. Trevisan]{Vilmar Trevisan}
\address{UFRGS - Universidade Federal do Rio Grande do Sul, 
Instituto de Matem\'atica, Porto Alegre, Brazil}\email{trevisan@mat.ufrgs.br }



\maketitle

\begin{abstract}
Given a tree $T$, let $q(T)$ be the minimum number of distinct eigenvalues in a symmetric matrix whose underlying graph is $T$. It is well known that $q(T)\geq d(T)+1$, where $d(T)$ is the diameter of $T$, and a tree $T$ is said to be diminimal if $q(T)=d(T)+1$. 
In this paper, we present families of diminimal trees of any fixed diameter. Our proof is constructive, allowing us to compute, for any diminimal tree $T$ of diameter $d$ in these families, a symmetric matrix $M$ with underlying graph $T$ whose spectrum has exactly $d+1$ distinct eigenvalues.
\end{abstract}


\section{Introduction}

As described by Chu in an influential survey paper~\cite{Chu98}, inverse eigenvalue problems are concerned with the reconstruction of a square matrix $M$ assuming that we are given total or partial information about its eigenvalues and/or eigenvectors. Chu points out to two fundamental questions associated with this problem: \begin{enumerate}
\item \emph{Solvability}, i.e., whether there exists a matrix $M$ with the required eigenvalues and/or eigenvectors. Such a matrix $M$ is said to be a realization of the inverse eigenvalue problem.

\item \emph{Computability}, i.e., whether, assuming that the problem has a solution, there is an efficient procedure to compute (or to find a numerical approximation) of a solution.  
\end{enumerate}
For the inverse eigenvalue problem to be nontrivial or to be meaningful in applications, it is often the case that the sought-after matrix $M$ needs to satisfy additional properties, that is, the domain must be restricted to matrices in a pre-determined class. 

In this paper, we consider classes of symmetric matrices that may be described in terms of graphs. Note that any symmetric matrix $M=(m_{ij}) \in \mathbb{F}^{n \times n}$ over a field $\mathbb{F}$ may be associated with a simple graph $G$ with vertex set $[n]=\{1,\ldots,n\}$ such that distinct vertices $i$ and $j$ are adjacent if and only if $m_{ij} \neq 0$. We say that $G$ is the \emph{underlying graph} of $M$. In fact, the matrix $M$ itself may be viewed as a weighted version of $G$, where each vertex $i$ is assigned the weight $m_{ii} \in \mathbb{F}$ and each edge $ij$ is assigned the weight $m_{ij} \in \mathbb{F}$. Often the focus is on matrices whose elements are in the field $\mathbb{R}$ of real numbers (or on Hermitian matrices over the field $\mathbb{C}$ of complex numbers). We refer to~\cite{BARRETT2020276}, and the references therein, for a more complete historical discussion of this type of inverse eigenvalue problem, known under the acronym IEPG, the \emph{inverse eigenvalue problem for a graph}.

Given a graph $G$, let $\mathcal{S}(G)$ and $\mathcal{H}(G)$ be the sets of real symmetric matrices and of complex Hermitian matrices whose underlying graph is $G$, respectively. An elementary fact about these matrices is that their eigenvalues are real numbers, so that, for any $n$-vertex graph $G$ and any matrix $M\in \mathcal{H}(G)$, the eigenvalues of $M$ may be written as $\lambda_1(M) \leq \cdots \leq \lambda_n(M)$\footnote{When the matrix $M$ is clear from context, we shall omit the explicit reference to $M$ and simply write $\lambda_1\leq \cdots \leq \lambda_n$.}. The multiset $\{\lambda_1,\ldots,\lambda_n\}$ is called the \emph{spectrum} of $M$ and is denoted by $\Spec(M)$, while $\DSpec(M)$ denotes the set of \emph{distinct} eigenvalues of $M$. The \emph{multiplicity} $m_M(\lambda)$ of $\lambda \in \mathbb{R}$ as an eigenvalue of $M$ is the number of occurrences of $\lambda$ in $\Spec(M)$\footnote{It will be convenient to write $m_M(\lambda)=0$ when $\lambda$ is not an eigenvalue of $M$.}. 

A class of matrices that has been under intense scrutiny is the class of \emph{acyclic symmetric matrices}, the class of matrices whose underlying graph is a connected acyclic graph (that is, a tree). The study of acyclic symmetric matrices may be traced back to Parter~\cite{Parter60} and Wiener~\cite{WIENER1984}, and there has been growing interest on properties of these matrices and of parameters associated with them starting with the systematic work of Leal Duarte, Johnson and their collaborators, see for instance~\cite{ParterWiener03,JOHNSON20027,DUARTE1989173,leal2002minimum}. One of the particularities of the acyclic case is that the inverse eigenvalue problem may be reduced to symmetric matrices, in the sense that, for any tree $T$, a multiset of real numbers is equal to the spectrum of a matrix in $\mathcal{H}(T)$ if and only if it is equal to the spectrum of a matrix in $\mathcal{S}(T)$ (see~\cite[Corollary 2.6.3]{JohnsonSaiago2018}). 

Our paper deals with the possible number of distinct eigenvalues of acyclic symmetric matrices. For an in-depth discussion of problems of this type, we refer to a comprehensive book on this topic by Johnson and Saiago~\cite{JohnsonSaiago2018}. More precisely, given a tree $T$, we wish to study the quantity
\begin{equation}\label{def_q}
q(T)=\min\{|\DSpec(A)| \colon A \in \mathcal{S}(T)\},
\end{equation}
the minimum number of distinct eigenvalues over all symmetric real matrices whose underlying graph is $T$. An easy lower bound on this number may be given in terms of the \emph{diameter} of $T$, which we now define. As usual, let $P_d$ denote a path on $d$ vertices. The distance $d_G(u,v)$ between two vertices $u$ and $v$ in a graph $G=(V,E)$ is the length (i.e. the number of edges $d-1$) of a shortest path $P_d$ connecting $u$ and $v$ in $G$, where we say that $d(u,v)=\infty$ if $u$ and $v$ lie in different components of $G$. The diameter $\diam(G)$ of $G$ is 
defined as 
$$\diam(G)=\max\{d(u,v) \colon u,v\in V\}\footnote{We observe that in~\cite{JohnsonSaiago2018,leal2002minimum} the value of the diameter corresponds to the number of \emph{vertices}, rather than edges, on the path connecting two vertices at maximum distance. As a consequence, what we call diameter $d$ is diameter $d+1$ in~\cite{JohnsonSaiago2018,leal2002minimum}}.$$
The following result is proved in~[Lemma 1]\cite{leal2002minimum}.
\begin{theorem}\label{thm:LB}
If $T$ is a tree with diameter $d$ and $A \in \mathcal{S}(T)$, then 
$q(T) \geq d+1$.
\end{theorem}
The authors of~\cite{leal2002minimum} suspected that, for every tree $T$ of diameter $d$, there exists a matrix $A \in \mathcal{S}(T)$ with exactly $d+1$ distinct eigenvalues. However, this turns out to be false. Barioli and Fallat~\cite{barioli2004two} constructed a tree $T$ on $16$ vertices such that $\diam(T)=6$, but $q(T)=8$. It is now known that $q(T)=d+1$ for every tree $T$ of diameter $d$ if and only if $d\leq 5$~\cite{JohnsonSaiago2018}. For diameter $d\geq 6$, it is thus natural to characterize the trees $T$ for which $q(T)=\diam(T)+1$, which are known as \emph{diameter minimal} (or \emph{diminimal}, for short). The set $\mathcal{D}_d$ of diminimal trees of any fixed dimension $d$ is nonempty, as we trivially have $P_{d+1} \in \mathcal{D}_d$. Johnson and Saiago~\cite{JohnsonSaiago2018} show that the families $\mathcal{D}_{d}$ are infinite\footnote{In the sense that the set of unlabelled trees of diameter $d$ that are diminimal is infinite.} for every $d$. 

One of the main tools used to address this problem in~\cite{JohnsonSaiago2018} is the construction of trees using an operation called \emph{branch duplication}~\cite{JOHNSON2020}. This concept will be formally defined in Section~\ref{sec:seeds},  but the intuition is that, for any fixed positive integer $d$, there is a finite set $\mathcal{S}_d$ of (unlabelled) trees of diameter $d$, called the \emph{seeds of diameter $d$}, with the property that any (unlabelled) tree of diameter $d$ may be obtained from one of the seeds of diameter $d$ by a sequence of branch duplications. As it turns out, for any tree $T$ of diameter $d$ there is a single seed of diameter $d$ from which it can be obtained, so that the seeds are precisely the trees that cannot be obtained from smaller trees through branch duplication. To illustrate why this can be useful for our purposes, we mention that Section 6.5 in~\cite{JohnsonSaiago2018} deals with $q(T)$ for trees of diameter $d=6$, for which the set $\mathcal{S}_6$ contains 12 seeds. Johnson and Saiago show that the families generated by nine of these seeds consist entirely of diminimal trees, while, in each of the remaining three families, at least one of the trees is not diminimal.

The main result in our paper is that, for any fixed $d \geq 4$, there are at least two seeds $S_d$ and $S'_d$ of diameter $d$ such that the families $\mathcal{T}(S_d)$ and $\mathcal{T}(S'_d)$ generated by these seeds consist entirely of diminimal trees. If $d \geq 5$ is odd, there is a third seed $S''_d$ for which this property holds. These seeds are formally defined in Definition~\ref{def_seeds}, and they are depicted in Figures~\ref{fig:s6}-\ref{fig:s72} for small values of $d$. 

\begin{theorem}\label{thm:main}
Let $d$ be a positive integer. Let $\mathcal{T}(S_d)$, $\mathcal{T}(S'_d)$ and $\mathcal{T}(S''_d)$ be the families of trees of diameter $d$ generated by the seeds $S_d$, $S'_d$ and $S''_d$, respectively, where $S'_d$ is defined for $d\geq 4$ and $S''_d$ for odd values of $d \geq 5$. For every $T\in \mathcal{T}(S_d) \cup \mathcal{T}(S'_d) \cup \mathcal{T}(S''_d)$, we have $q(T)=d+1$.
\end{theorem}

The main additional tool in our proof of Theorem~\ref{thm:main} is an algorithm by Jacobs and Trevisan~\cite{JT2011} that was proposed to solve a problem known as \emph{eigenvalue location} for matrices associated with graphs. A detailed discussion is deferred to Section~\ref{sec:eigenvalue_location}, but we can anticipate that it will have an important role in an inductive approach for Theorem~\ref{thm:main}.

As a byproduct of our proof of Theorem~\ref{thm:main} (see Theorem~\ref{main_th}), we obtain a constructive procedure that, given a tree $T \in \mathcal{T}(S_d) \cup \mathcal{T}(S'_d) \cup \mathcal{T}(S''_d)$, produces a symmetric matrix $A \in \mathbb{R}^{n \times n}$ with underlying tree $T$ with the property that $q(T)=|\DSpec(A)|=d+1$. This means that, in addition to exploring the existence of such a matrix, we also address its computability. In particular, the procedure allows us to produce such a matrix $A$ with integral spectrum, i.e., with the property that its spectrum consists entirely of integers. More generally, let
\begin{eqnarray*}
&& q_{int}(T)=\min\{|\DSpec(A)| \colon A \in \mathcal{S}_{int}(T)\},
\end{eqnarray*}
where $\mathcal{S}_{int}(T)$ is the subset of $\mathcal{S}(T)$ whose matrices are integral (i.e., have integral spectrum). Clearly, we have $q(T)\leq q_{int}(T)$ for any tree $T$.
Our work implies that the following holds for all trees $T$ with diameter $d \leq 5$ and for all $T \in \mathcal{T}(S_d) \cup \mathcal{T}(S'_d) \cup \mathcal{T}(S''_d)$ with $d \geq 6$:
$$q(T)=q_{int}(T).$$ 
It would be interesting to understand how these parameters relate for arbitrary trees. We should mention that, even though the focus of this paper is on matrices associated with trees, there has been a lot of research on the parameter $q(G)$ for more general graphs, see \cite{F2013,allred2022combinatorial,BARRETT2020276,barrett2017generalizations,fallat2022minimum} for example.

Our paper is structured as follows. In the next two sections, we describe the main ingredients used in our proofs. In Section~\ref{sec:eigenvalue_location}, we present an algorithm that \emph{locates} eigenvalues of trees, that is, an algorithm that for any real symmetric matrix $M$ whose underlying graph is a tree and for any given real interval $I$, finds the number of eigenvalues of $M$ in $I$. We illustrate its usefulness by providing a short proof of the classical Parter-Wiener Theorem.
In Section~\ref{sec:seeds}, we describe how trees of any fixed diameter $d$ can be constructed by a sequence of \emph{branch duplications} starting with some irreducible tree with diameter $d$, which is known as a seed.

The remaining sections deal with the proof of Theorem~\ref{thm:main}. In Section~\ref{sec:strongly_realizable}, we state a technical tool (Theorem~\ref{main_th}) that is the heart of the proof. Given $d\geq 1$, it will allow us to inductively define a set of real numbers (of size $d+1$) that, for any tree $T \in \mathcal{T}(S_d)$ with diameter $d$, is equal to the set of distinct eigenvalues in the spectrum of a symmetric matrix $M(T)$ with underlying graph $T$. A set of this type will be called \emph{strongly realizable} because there are realizations of it for all trees in the class under consideration. The existence of such a set immediately implies the validity of Theorem~\ref{thm:main} for seeds of type $S_d$.

Theorem~\ref{main_th} will then be proved by induction in Section~\ref{sec:proof_technical}. To conclude the paper, Section~\ref{sec:other_seeds} uses strongly realizable sets to give a proof of Theorem~\ref{thm:main} for seeds of type $S_d'$ and $S_d''$. Moreover, we explain how this may be used to obtain matrices with the minimum number of distinct eigenvalues that satisfy additional properties, such as having integral spectrum. An explicit construction is given in Section~\ref{sec:example}.  



\section{Eigenvalue location in trees}\label{sec:eigenvalue_location}

In a seminal paper~\cite{JT2011}, Jacobs and one of the current authors have proposed an algorithm that, given a real symmetric matrix $M$ whose underlying graph is a tree and a real interval $I$, finds the number of eigenvalues of $M$ in $I$. In fact, the work in~\cite{JT2011} was specifically concerned with eigenvalues of the adjacency matrix of an arbitrary tree. However, the strategy could be extended in a natural way to arbitrary symmetric matrices associated with trees. This more general algorithm, stated in Figure~\ref{chap3:treeunder}, appears in~\cite{TEMA1041}. 

The algorithm runs on a \emph{rooted tree}, that is a tree $T$ for which one of the vertices $r$ is distinguished as the \emph{root}. Each neighbor of $r$ is regarded as a {\em child} of $r$, and $r$ is called its {\em parent}. For each child $c$ of $r$, all of its neighbors, except the parent, become its children. This process continues until all vertices except $r$ have parents.
A vertex that does not have children is called a {\em leaf} of the rooted tree. For the algorithm, the tree $T$ that underlies the input matrix $M$ may be assigned an arbitrary root, but its vertex set must be ordered \emph{bottom-up}, that is, any vertex must appear after all its children. In particular, the root is the last vertex in such an ordering.

\begin{figure}[h]
\input{chap3-fig-under}
\caption{\label{chap3:treeunder} Diagonalizing $M + xI$ for a symmetric 
matrix $M$ associated with the tree $T$.}
\end{figure}
\index{Diagonalize Weighted Tree}

The following theorem summarizes the way in which the algorithm will be applied. Its proof is based on a property of matrix congruence known as Sylvester's Law of Inertia, we refer to~\cite{livro} for details.
\begin{theorem}
\label{inertia}
Let $M$ be a symmetric matrix of order $n$ that corresponds to a weighted tree $T$ and let $x$ be a real number. Given a bottom-up ordering of $T$, let $D$ be the diagonal matrix produced by Algorithm Diagonalize with entries $T$ and $x$. The following hold:
\begin{itemize}
    \item[(a)] The number of positive entries in the diagonal of $D$ is the number of eigenvalues of $M$ (including multiplicities) that are greater than $-x$. 
    
    \item[(b)] The number of negative entries in the diagonal of $D$ is the number of eigenvalues of $M$ (including multiplicities) that are less than $-x$. 
    
    \item[(c)] The number of zeros in the diagonal of $D$ is the multiplicity of $-x$ as en eigenvalue of $M$. 
\end{itemize}
\end{theorem}

An immediate consequence of this result is the well-known fact that the multiplicity of the maximum and of the minimum eigenvalue of any tree is always equal to 1.
\begin{theorem}\label{thm:simpleroots}
Let $T$ be a tree, let $M\in \mathcal{S}(T)$, and consider $\lambda_{\min}=\min(\Spec(M))$ and $\lambda_{\max}=\max(\Spec(M))$. Then, $m_{M}(\lambda_{\min}) = 1 =m_{M}(\lambda_{\max})$.
\end{theorem}

\begin{proof}
Let $T$ be a tree and $M\in \mathcal{S}(T)$. We prove the theorem for $\lambda_{\min}$, the proof for $\lambda_{\max}$ follows from the fact that $\lambda_{\max}(M)=-\lambda_{\min}(-M)$. 

Choose an arbitrary root $v$ for $T$ and fix a bottom-up ordering $v_1,\ldots,v_n=v$ of $T$. Set $x=-\lambda_{\min}$ and consider an application of \texttt{Diagonalize}$(M,x)$. Because $\lambda_{\min}$ is an eigenvalue of $M$, at least one of the diagonal elements $d_j$ must be zero at the end of the algorithm by Theorem~\ref{inertia}(c). 

We claim that $v_j=v$, which implies the desired result. Assume for a contradiction that $v_j \neq v$, so that $v_j$ has a parent $v_k$ in $T$. Because $d_j$ is 0, at the time $v_k$ is processed, it has a child with value 0. The algorithm assigns the positive value 2 to one of the children $v_{j'}$ of $v_k$ with this property and the negative value $-(m_{j'k})^2/2$ to $v_k$. These values cannot be modified in later steps. Theorem~\ref{inertia}(b) implies that $M$ has an eigenvalue $\lambda$ such that $\lambda<\lambda_{\min}$, a contradiction. 
\end{proof}

Theorem~\ref{inertia} can also be used to give a short proof of a result attributed to Parter and Wiener, see~\cite[Theorem 2.3.1]{JohnsonSaiago2018}. We include the proof here to illustrate how our proof method applies. Given a tree $T$ and a matrix $M(T)$ for which $T$ is the underlying tree, we write $M[T-v]$ to denote the submatrix of $M$ obtained by deleting the row and the column corresponding to a vertex $v$ of $T$. More generally, if $T'$ is a subgraph of $T$, we write $M[T']$ for the submatrix of $M$ induced by the rows  and columns corresponding to vertices of $T'$. 
\begin{theorem}[Parter-Wiener Theorem]\label{parter-wiener}
 Let $T$ be a tree, let $M\in \mathcal{S}(T)$, and suppose that $\lambda \in \mathbb{R}$ is such that $m_{M}(\lambda)\geq2$. Then there is a vertex $v$ of $T$ of degree at least $3$ such that $m_{M[T-v]}(\lambda) = m_{M}(\lambda) + 1$. Moreover, $\lambda$ occurs as an eigenvalue of $M[T_i]$ for at least three different components $T_i$ of $T-v$.
\end{theorem}

\begin{proof}
Let $T$ be an $n$-vertex tree, $M\in \mathcal{S}(T)$ and suppose that $\lambda \in \mathbb{R}$ is such that $m_{M}(\lambda)\geq2$.

Choose some vertex $v_n$ of $V(T)$ as the root of $T$ and set $x=-\lambda$. Consider an application of \texttt{Diagonalize}$(M,x)$ with root $v_n$. By Theorem~\ref{inertia}(c), at least two diagonal elements of the output matrix $D$ must be 0. Fix a vertex $v_j$ that is farthest from the root such that $d_j=0$, so that $j<n$. Let $v=v_k$ be the parent of $v_j$. Let $T_i$ be the components of $T-v_k$ rooted in each of the children $v_{k,i}$ of $v_k$, where $i \in \{1,\ldots,\ell\}$. If $v_k\neq v_n$, let $T_0$ be the component of $T-v_k$ that contains $v_n$ (and assume it is rooted at $v_n$).

First note that $\ell \geq 2$. Indeed, if $\ell=1$, then $v_j$ would be the only child of $v_k$. However, since $d_j=0$, when the algorithm processes $v_k$, it redefines $d_j$ as 2, contradicting our assumption. In fact, this argument further implies that $v_j$ must have at least one sibling $v_{j'}$ to which the algorithm assigns value $0$ as it processes $v_{j'}$, but then redefines it as 2. 

Consider applications of \texttt{Diagonalize}$(M[T_i],x)$ for $i \in \{1,\ldots,\ell\}$. Our assumption about the distance from $v_j$ to the root implies that 0 can only appear (as the final value) at the root of each such application. On the other hand, the previous paragraph ensures that 0 appears as the final value of at least two of the roots, namely $v_j$ and $v_{j'}$. 

If 0 is the value of at least three of these roots, we conclude that $v_k$ has degree at least three and that $\lambda$ occurs as an eigenvalue of at least three components of $T-v$ by Theorem~\ref{inertia}(b). 

If 0 appears in exactly two of the roots, we conclude that $v_k\neq v_n$, as otherwise one of the 0's would be redefined as 2 when processing $v_n$, and $v_n$ would be assigned a negative value, contradicting the assumption that $m_{M}(\lambda) \geq 2$. The same considerations imply that, when performing \texttt{Diagonalize}$(M,x)$, there are initially two occurrences of 0 at the children of $v_k$, but, when the algorithm processes $v_k$, it replaces one of the zeros by 2, $v_k$ gets a negative value and the edge between $v_k$ and its parent is deleted. Because of this, the values assigned by \texttt{Diagonalize}$(M,x)$ to the remaining vertices of $T$ are not affected by the values on $v_k$'s branch, that is, they are exactly the values assigned by \texttt{Diagonalize}$(M[T_0],x)$. In particular, 0 must appear at least once at the output of \texttt{Diagonalize}$(M[T_0],x)$, thus $\lambda$ is an eigenvalue of $T_0$. Overall, $\lambda$ occurs as an eigenvalue of at least three components of $T-v$.

To conclude the proof, we still need to establish $m_{M[T-v]}(\lambda) = m_M(\lambda) + 1$, where $v=v_k$. But this follows immediately from the argument above. Let $s$ be the number of times that 0 appears at children of $v_k$ (before $v_k$ is processed). After processing $v_k$, one of the zeros becomes 2 and the edge between $v_k$ and its parent (if it exists) is deleted, so that $m_{M}(\lambda)=(s-1)+m_{M[T_0]}(\lambda)$ if $v_k\neq v_n$, and $m_{M}(\lambda)=s-1$ if $v_k=v_n$. On the other hand, $m_{M[T-v]}(\lambda)=s+m_{M[T_0]}(\lambda)$ if $v_k\neq v_n$, and $m_{M[T-v]}(\lambda)=s$ if $v_k=v_n$.
\end{proof}

The following result is proved with similar ideas.
\begin{lemat}\label{proposition}
    Let $T$ be a tree and $M\in \mathcal{S}(T)$. If $v\in V(T)$ is such that $m_{M[T-v]}(\lambda)=m_{M}(\lambda)+1$, for some $\lambda\in\mathbb{R}$, then the following holds when Algorithm Diagonalize is performed for $M$ and $-\lambda$ with root $v$. There is a child $v_j$ of $v$ such that, after processing $v_j$, the algorithm assigns value $d_j=0$.
\end{lemat}

\begin{proof}
Let $T$ be a rooted tree and $M\in \mathcal{S}(T)$. Let $v\in V(T)$ be such that $m_{M[T-v]}(\lambda)=m_{M}(\lambda)+1$, for some $\lambda\in\mathbb{R}$. Consider $T$ rooted at $v$. Let $T_1,\ldots,T_p$ be the connected components of $T-v$ rooted at the children $v_1,\ldots,v_p$ of $v$. 
By Theorem~\ref{inertia}, $m_{M}(\lambda)$ is given by the number of occurrences of $0$ in the diagonal of the matrix produced by \texttt{Diagonalize}$(M,-\lambda)$ with root $v$. Similarly, $m_{M[T-v]}(\lambda)$ is the sum of the number of occurrences of $0$ in the diagonal of the matrices $D_i$ produced by \texttt{Diagonalize}$(M[T_i],-\lambda)$ with root $v_i$. By hypothesis, this sum is larger than $m_{M}(\lambda)$. In particular, one of the 0´s assigned by \texttt{Diagonalize}$(M[T_i],-\lambda)$ must lie on a vertex $u$ that is assigned a nonzero value by \texttt{Diagonalize}$(M,-\lambda)$. 

On the other hand, the value assigned by \texttt{Diagonalize}$(M[T_i],-\lambda)$ with root $v_i$ to a vertex $w\neq v_i$ is precisely the value assigned to $w$ by \texttt{Diagonalize}$(M,-\lambda)$ with root $v$. As a consequence, the vertex $u$ mentioned in the previous paragraph must be $v_j$ for some $j\in\{1,\ldots,p\}$. This means that, in an application of \texttt{Diagonalize}$(M,-\lambda)$ with root $v$, the algorithm assigns value $d_j=0$ to $v_j$ upon processing $v_j$, and later redefines the value of $d_j$ as 2 when processing its parent $v$.
\end{proof}

\section{Trees of diameter $d$ and branch decompositions} \label{sec:seeds}

In this section, we shall describe an operation known as branch decomposition, which allows us to view trees of diameter $d$ as being generated by a finite number of such trees, which are known as seeds. 

Let $d \geq 1$ be a fixed integer and let $\mathcal{T}_{d}^{(n)}$ be the set of $n$-vertex trees with diameter $d$, where $n \geq 3$. Given a tree $T \in \mathcal{T}_{d}^{(n)}$, there is a natural way to consider it as a rooted tree. 
\begin{defn}[Main root]\label{mainrooteven}
Let $T=(V,E)$ be a tree with diameter $d$. 
\begin{itemize}
    \item[(a)] If $d=2k$ for some $k\in\mathbb{N}$, then $v\in V$ is the \emph{main root} of $T$ if it is the central vertex of a maximum path $P_{2k+1}$ in $T$.
    
    \item[(b)] If $d=2k+1$ for some $k\in\mathbb{N}$, then $e\in E$ is the \emph{main edge} of $T$ if it is the central edge of a maximum path $P_{2k+2}$ in $T$. Each endpoint of $e$ is called a \emph{main root} of $T$.
\end{itemize}
\end{defn}

We note that the main root and the main edge are well defined. For (a), observe that any two distinct copies $Q_1$ and $Q_2$ of $P_{2k+1}$ in $T$ must intersect in a vertex $v$ that is the central vertex of both, otherwise the path $Q$ created by merging the two longest subpaths of $Q_1$ and $Q_2$ joining $v$ to a leaf would have more than $2k+1$ edges, a contradiction. We may similarly argue that any two longest paths in a tree with odd diameter share their central edge.


To prove Theorem~\ref{thm:main}, we will construct classes of trees of diameter $d$ in a recursive way. To this end, we define an operation on rooted trees. Let $p\geq 1$ and consider disjoint trees $T_0,T_1,\ldots,T_p$ rooted at vertices $v_0,v_1,\ldots,v_p$, respectively. We write $T_0\odot (T_1,\ldots,T_p)$ for the tree with vertex set $V=\bigcup_{\ell=0}^{p}V(T_\ell)$ and edge set $E=\cup_{\ell=1}^{p}\{v_0v_\ell\}\cup\bigcup_{\ell=0}^{p}E(T_\ell)$ and we write $T= T_0 \odot (T_1,\ldots,T_p)$ (see Figure~\ref{fig:odot} for $p=3$). If $p=1$, we simplify the notation to $T_0\odot (T_1)= T_0 \odot T_1$.

\newcommand{\Ttres}{
\begin{tikzpicture}[scale=.7,auto=left,every node/.style={circle,scale=0.5}]

\path( 0,0)node[shape=circle,draw,fill=black] (1) {}
      (.5,0)node[shape=circle,draw,fill=black] (2) {}
      (1.5,0)node[shape=circle,draw,fill=black] (3) {}
      (1,1)node[shape=circle,draw,fill=black] (4) {}
      (0,1)node[shape=rectangle,label=left:\Large{$v_{3}$},draw,fill=red] (5) {};
      
      \draw[-](2)--(4);
      \draw[-](3)--(4);
      \draw[-](5)--(4);
      \draw[-](5)--(1);

\end{tikzpicture}
}

\newcommand{\Tzero}{
\begin{tikzpicture}[scale=.7,auto=left,every node/.style={circle,scale=0.5}]

\path( 0,0)node[shape=circle,draw,fill=black] (1) {}
      (1,0)node[shape=circle,draw,fill=black] (2) {}
      (0,.7)node[shape=circle,draw,fill=black] (3) {}
      (1,.7)node[shape=circle,draw,fill=black] (4) {}
      (.5,1.3)node[shape=rectangle,label=left:\Large{$v_{0}$},draw,fill=red] (5) {}
      (.5,2)node[shape=circle,draw,fill=black] (6) {};

      \draw[-](1)--(3);
      \draw[-](3)--(5);
      \draw[-](2)--(4);
      \draw[-](4)--(5);
      \draw[-](5)--(6);

\end{tikzpicture}
}
\newcommand{\Tdois}{
\begin{tikzpicture}[scale=.7,auto=left,every node/.style={circle,scale=0.5}]

\path( 0,0)node[shape=circle,draw,fill=black] (1) {}
      (1,.5)node[shape=rectangle,label=right:\Large{$v_{2}$},draw,fill=red] (2) {};       

      \draw[-](1)--(2);
          
     \end{tikzpicture}
}
\newcommand{\Tum}{
\begin{tikzpicture}[scale=.7,auto=left,every node/.style={circle,scale=0.5}]

\path( 0,0)node[shape=circle,draw,fill=black] (1) {}
      (.5,.5)node[shape=rectangle,label=left:\Large{$v_{1}$},draw,fill=red] (2) {}
      (1,1)node[shape=circle,draw,fill=black] (3) {};

      \draw[-](1)--(2);
      \draw[-](2)--(3);

\end{tikzpicture}
}

\newcommand{\operation}{
\begin{tikzpicture}[scale=.7,auto=left,every node/.style={circle,scale=0.5}]

\path( 0,0)node[shape=circle,draw,fill=black] (1) {}
      (.6,0)node[shape=circle,draw,fill=black] (2) {}
      (.3,1)node[shape=circle,label=left:\Large{$v_{1}$},draw,fill=black] (3) {} 

      (1.5,0)node[shape=circle,draw,fill=black] (4) {}
      (1.5,1)node[shape=circle,label=left:\Large{$v_{2}$},draw,fill=black] (5) {}  

      (2.5,0)node[shape=circle,draw,fill=black] (6) {}
      (2.5,1)node[shape=circle,label=left:\Large{$v_{3}$},draw,fill=black] (7) {} 
      (3.2,0)node[shape=circle,draw,fill=black] (8) {}      
      (3.5,1)node[shape=circle,draw,fill=black] (9) {}
      (3.7,0)node[shape=circle,draw,fill=black] (10) {}

      (1.5,2)node[shape=rectangle,label=left:\Large{$v_{0}$},draw,fill=red] (0) {}
      (1.2,3)node[shape=circle,draw,fill=black] (11) {}
      (2.2,2.5)node[shape=circle,draw,fill=black] (14) {}
      (3,2.5)node[shape=circle,draw,fill=black] (15) {}
      (2.2,3.5)node[shape=circle,draw,fill=black] (12) {}
      (3,3.5)node[shape=circle,draw,fill=black] (13) {};

      \draw[-](1)--(3);
      \draw[-](2)--(3);

      \draw[-](4)--(5);

      \draw[-](6)--(7);
      \draw[-](7)--(9);
      \draw[-](8)--(9);
      \draw[-](9)--(10);

      \draw[-](0)--(11);
      \draw[-](0)--(14);
      \draw[-](0)--(12);
      \draw[-](14)--(15);
      \draw[-](12)--(13);

      \draw[dotted](0)--(3);
      \draw[dotted](0)--(5);
      \draw[dotted](0)--(7);
\end{tikzpicture}
}


\captionsetup[subfigure]{labelformat=empty}

\begin{figure}
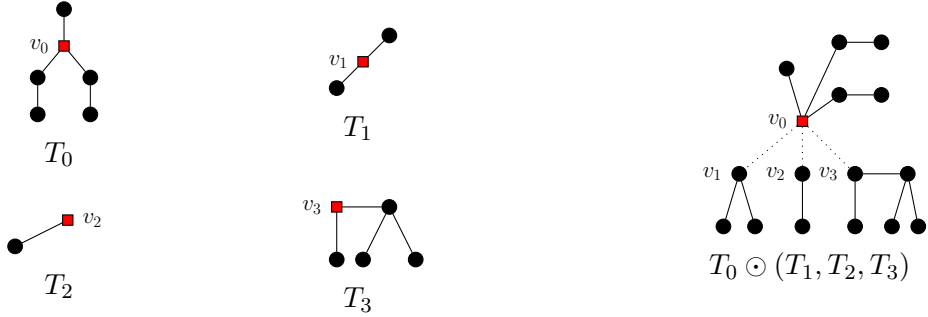

\centering
\begin{minipage}[h]{.5\textwidth}
\centering

 \begin{subfigure}{0.49\textwidth}
     \centering
     \Tzero
     \caption{$T_{0}$}
     \label{fig:a}
 \end{subfigure}
 \hfill
 \begin{subfigure}{0.49\textwidth}
     \centering
     \Tum
     \caption{$T_{1}$}
     \label{fig:b}
 \end{subfigure}
 
 \medskip
 \begin{subfigure}{0.49\textwidth}
     \centering
     \Tdois
     \caption{$T_{2}$}
     \label{fig:c}
 \end{subfigure}
 \hfill
 \begin{subfigure}{0.49\textwidth}
     \centering
     \Ttres
     \caption{$T_{3}$}
     \label{fig:d}
 \end{subfigure}

\end{minipage}\hfill
\begin{minipage}[h]{.5\textwidth}

 \begin{subfigure}{1.05\textwidth}
     \centering
     \operation
     \caption{$T_0\odot (T_1,T_2,T_3)$}
     \label{fig:e}
 \end{subfigure}
\end{minipage}

\caption{The rooted tree $T_0\odot (T_1,T_2,T_3)$. The root of each tree is denoted by a square. \label{fig:odot}}
\end{figure}

The \emph{height} $h(T)$ of a rooted tree $T$ is the distance of the root $v$ to the farthest vertex in $T$, i.e., $h(T)=\max\{d(v,u):u\in V(T)\}$. Note that, when a tree $T$ of diameter $d$ is rooted at a main root, then its height is $h(T)=\ceil{\frac{d}{2}}$.

As mentioned in the introduction, the authors of the book~\cite{JohnsonSaiago2018} consider families of trees constructed by successive applications of operations called \emph{branch duplications}. Given a tree $T$, we say that $T_j$ is a \emph{branch of $T$ at a vertex $v$} if $T_j$ is a component of $T-v$. We can view the branch as a rooted tree with root given by the neighbor of $v$.  An \emph{$s$-combinatorial branch duplication (CBD)} of $T_j$ at $v$ results in a new tree where $s$ copies of $T_j$ are appended to $T$ at $v$ (see Figure
~\ref{fig:unfolding}). A tree $T'$ that is obtained from $T$ by a finite sequence of CBDs is called an \emph{unfolding} of $T$. In this case, we also say that $T$ is a \emph{folding} of $T'$. It is easy to see that for $T$ to be an unfolding of some other tree, then $T$ must contain a vertex $v$ such that $T-v$ has at least two isomorphic branches, by which we mean that there is a root-preserving isomorphism between the two branches. 

\begin{figure}[H]
\centering
\begin{minipage}[h]{.5\textwidth}
\centering 
\begin{subfigure}{0.49\textwidth}
\centering 
     \begin{tikzpicture}[scale=.7,auto=center,every node/.style={circle,scale=0.5}]

\path(0,-.2)node[shape=circle,draw,fill=black] (1) {}
      (1,-.2)node[shape=circle,draw,fill=black] (2) {}
      (.5,.5)node[shape=circle,label=left:\Large{$v_{1}$},draw,fill=black] (3) {}
      (.5,1.5)node[shape=rectangle,label=left:\Large{$v$},draw,fill=red] (0) {}
      (.5,2.2)node[shape=circle,draw,fill=black] (4) {}
      (1.2,1.5)node[shape=circle,draw,fill=black] (5) {}
      (1.9,1.5)node[shape=circle,draw,fill=black] (6) {};

     \draw[-](1)--(3);
     \draw[-](2)--(3);
     \draw[-](3)--(0);
     \draw[-](4)--(0);
     \draw[-](5)--(0);
     \draw[-](5)--(6);

\end{tikzpicture}
     \caption{Tree $T$ of diameter $4$ rooted at $v$.}
 \end{subfigure}
 
 \end{minipage}\hfill
\begin{minipage}[h]{.5\textwidth}
\centering

 \begin{subfigure}{0.49\textwidth}
     \centering
\begin{tikzpicture}[scale=.7,auto=left,every node/.style={circle,scale=0.5}]

\path(0,-.2)node[shape=circle,draw,fill=black] (1) {}
      (1,-.2)node[shape=circle,draw,fill=black] (2) {}
      (.5,.5)node[shape=circle,label=left:\Large{$v_{1}$},draw,fill=black] (3) {}
      (.5,1.5)node[shape=rectangle,label=left:\Large{$v$},draw,fill=red] (0) {}
      (.5,2.2)node[shape=circle,draw,fill=black] (4) {}
      (1.2,1.5)node[shape=circle,draw,fill=black] (5) {}
      (1.9,1.5)node[shape=circle,draw,fill=black] (6) {}

      (1.5,-.2)node[shape=circle,draw,fill=black] (7) {}
      (2.5,-.2)node[shape=circle,draw,fill=black] (8) {}
      (2,.5)node[shape=circle,draw,fill=black] (9) {}

      (3,-.2)node[shape=circle,draw,fill=black] (10) {}
      (4,-.2)node[shape=circle,draw,fill=black] (11) {}
      (3.5,.5)node[shape=circle,draw,fill=black] (12) {};

     \draw[-](1)--(3);
     \draw[-](2)--(3);
     \draw[-](3)--(0);
     \draw[-](4)--(0);
     \draw[-](5)--(0);
     \draw[-](5)--(6);

     \draw[-](7)--(9);
     \draw[-](8)--(9);
     \draw[dotted](9)--(0);

     \draw[-](11)--(12);
     \draw[-](10)--(12);
     \draw[dotted](12)--(0);

\end{tikzpicture}     
\caption{$2$-CBD of $T_1$ (the branch of $T-v$ that contains $v_1$) at $v$.}
 \end{subfigure}
 
 \end{minipage}
\caption{An unfolding of a tree $T$.}\label{fig:unfolding}

\end{figure}
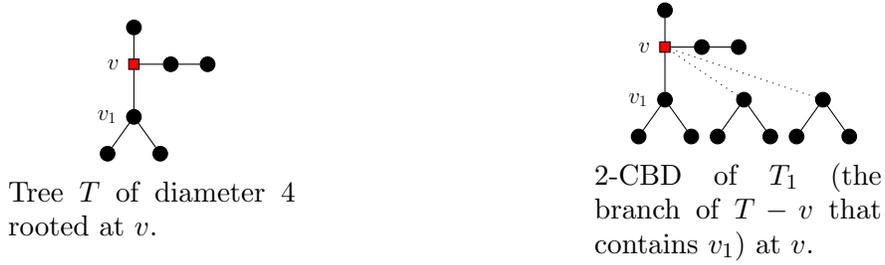

In this paper (as was the case in~\cite{JohnsonSaiago20182018}), we are interested in unfoldings that preserve the diameter. For this reason, a CBD will be only performed on branches of $T-v$ that do not contain any main root of $T$ (as the diameter would increase otherwise). An (unlabelled) tree $T$ of diameter $d$ is said to be a \emph{seed} if it cannot be folded into a smaller tree of diameter $d$. The work in~\cite{JOHNSON2020} shows that, for every positive integer $d$, there is a finite number of seeds of diameter $d$. Moreover, every tree of diameter $d$ is an unfolding of precisely one of these seeds. For example, the path $P_{d+1}$ is the only seed for trees of diameter $d\leq 3$. For diameter $4$ and $5$ there are two and three seeds, respectively, and for $d=6$ there are twelve seeds. 

Note that we can think of unfolding in terms of the operation $\odot$. Let $T$ be a tree and consider a branch $T_j$ of $T$ at a vertex $v$. Let $v_j$ be the root of $T_j$ (i.e. the neighbor of $v$ in $T_j$). Define $T_0=T-T_j$, rooted at $v$. Let $T_j^{(i)}$ be disjoint copies of $T_j$, for $i\in\{1,\ldots,s\}$, whose roots are denoted $v^{(i)}_j$, respectively. It is clear that $T_0\odot(T_j,T^{(1)}_j,\ldots,T^{(s)}_j)$ is an $s$-CBD of $T_j$ at $v$.

As mentioned in the introduction, given $d$, we are interested in three families of trees of diameter $d$, namely the trees generated by seeds $S_d$, $S'_d$ and $S''_d$. We formally define them here in terms of the operation $\odot$. In the definition, it is convenient to construct the seeds as trees that are rooted at a central vertex.
\begin{defn}\label{def_seeds}
Let $S_0=K_1$, $S_1=P_2$ and $S_2=P_3$ be the only trees with a single vertex, two vertices and three vertices, respectively, and consider them as rooted trees such that $S_2$ is rooted at the central vertex. Let $S'_{4}= S_{0}\odot(S_{1},S_{1})$, $S'_{5}= (S_{0}\odot S_{1})\odot(S_{0}\odot S_{1})$ and $S''_{5}= (S_{0}\odot S_{1})\odot(S_{1}\odot S_{1})$.
For $k \geq 2$, define
\begin{itemize}
    \item[(i)] $S_{2k-1}=S_{2k-3}\odot S_{2k-3}$ and $S_{2k}=S_{2k-3}\odot (S_{2k-3},S_{2k-3})$;
    
    \item[(ii)] $S'_{2k+2}= S_{2k-3}\odot(S_{2k-1},S_{2k-1})$ and $S'_{2k+3}= (S_{2k-3}\odot S_{2k-1})\odot(S_{2k-3}\odot S_{2k-1})$;
    
    \item[(iii)] $S''_{2k+3}= (S_{2k-3}\odot S_{2k-1})\odot(S_{2k-1}\odot S_{2k-1})$.
\end{itemize}
\end{defn}

We observe that $S_3$ is the only seed of diameter three, that $S_4$ and $S_4'$ are the only two seeds of diameter four and that $S_5$, $S_5'$ and $S_5''$ are the only three seeds of diameter five. In Figure~\ref{fig:seeds}, we depict $S_d$, $S_d'$ and $S_d''$ for $d\in \{6,7\}$.


\newcommand{\seedseis}{
\begin{tikzpicture}[scale=.7,auto=left,every node/.style={circle,scale=0.3}]

\path( 0,0)node[shape=circle,draw,fill=black] (7) {}
     (0,.5)node[shape=circle,draw,fill=black] (6) {}
     (0,1)node[shape=rectangle,draw,fill=black] (4) {}
     (-.5,1)node[shape=circle,draw,fill=black] (5) {}

       (1,1)node[shape=circle,draw,fill=black] (9) {}
      (.5,1)node[shape=rectangle,draw,fill=black] (8) {}
      (.5,0.5)node[shape=circle,draw,fill=black] (10) {}
     (.5,0)node[shape=circle,draw,fill=black] (11) {}

      (.25,2)node[shape=rectangle,draw,fill=red] (0) {}
      (.25,2.5)node[shape=circle,draw,fill=black] (1) {}
      (.75,2)node[shape=circle,draw,fill=black] (2) {}
     (1.25,2)node[shape=circle,draw,fill=black] (3) {};
     
      
      \draw[-](7)--(6);
      \draw[-](6)--(4);
      \draw[-](4)--(5);

      \draw[-](11)--(10);
      \draw[-](10)--(8);
      \draw[-](8)--(9);

      \draw[-](0)--(1);
      \draw[-](0)--(2);
      \draw[-](2)--(3);

     \draw[dotted](0)--(4);
      \draw[dotted](0)--(8);

     \end{tikzpicture}
}


\newcommand{\seedseisum}{
\begin{tikzpicture}[scale=.7,auto=left,every node/.style={circle,scale=0.3}]

\path( 0,0)node[shape=circle,draw,fill=black] (7) {}
     (0,.5)node[shape=circle,draw,fill=black] (6) {}
     (0,1)node[shape=rectangle,draw,fill=black] (4) {}
     (-.5,1)node[shape=circle,draw,fill=black] (5) {}

       (1,1)node[shape=circle,draw,fill=black] (9) {}
      (.5,1)node[shape=rectangle,draw,fill=black] (8) {}
      (.5,0.5)node[shape=circle,draw,fill=black] (10) {}
     (.5,0)node[shape=circle,draw,fill=black] (11) {}

      (.25,2)node[shape=rectangle,draw,fill=red] (0) {}
      (.25,2.5)node[shape=circle,draw,fill=black] (1) {};
     
      
      \draw[-](7)--(6);
      \draw[-](6)--(4);
      \draw[-](4)--(5);

      \draw[-](11)--(10);
      \draw[-](10)--(8);
      \draw[-](8)--(9);

      \draw[-](0)--(1);

     \draw[dotted](0)--(4);
      \draw[dotted](0)--(8);

     \end{tikzpicture}
}


\newcommand{\seedsete}{
\begin{tikzpicture}[scale=.7,auto=left,every node/.style={circle,scale=0.3}]

\path( 0,.5)node[shape=circle,draw,fill=black] (11) {}
     (0.5,.5)node[shape=circle,draw,fill=black] (10) {}
     (1,.5)node[shape=rectangle,draw,fill=black] (8) {}
     (1,0)node[shape=circle,draw,fill=black] (9) {}
     (1.5,0.5)node[shape=rectangle,draw,fill=black] (12) {}
     (1.5,0)node[shape=circle,draw,fill=black] (13) {}
      (2,0.5)node[shape=circle,draw,fill=black] (14) {}
      (2.5,0.5)node[shape=circle,draw,fill=black] (15) {}

      (0,1.5)node[shape=circle,draw,fill=black] (1) {}
     (0.5,1.5)node[shape=circle,draw,fill=black] (2) {}
     (1,1.5)node[shape=rectangle,draw,fill=red] (0) {}
     (1,2)node[shape=circle,draw,fill=black] (3) {}
     (1.5,1.5)node[shape=rectangle,draw,fill=black] (5) {}
     (1.5,2)node[shape=circle,draw,fill=black] (4) {}
      (2,1.5)node[shape=circle,draw,fill=black] (6) {}
      (2.5,1.5)node[shape=circle,draw,fill=black] (7) {};

      \draw[-](11)--(10);
      \draw[-](10)--(8);
      \draw[-](8)--(9);

      \draw[-](12)--(13);
      \draw[-](12)--(14);
      \draw[-](14)--(15);

       \draw[-](8)--(12);

      \draw[-](1)--(2);
      \draw[-](2)--(0);
      \draw[-](0)--(3);

      \draw[-](0)--(5);
      \draw[-](5)--(4);
      \draw[-](5)--(6);

     \draw[-](6)--(7);
     \draw[dotted](0)--(8);

     \end{tikzpicture}
}


\newcommand{\seedseteum}{
\begin{tikzpicture}[scale=.7,auto=left,every node/.style={circle,scale=0.3}]

\path
     (0.5,.5)node[shape=circle,draw,fill=black] (10) {}
     (1,.5)node[shape=rectangle,draw,fill=black] (8) {}
     (1,0)node[shape=circle,draw,fill=black] (9) {}
     (1.5,0.5)node[shape=rectangle,draw,fill=black] (12) {}
     (1.5,0)node[shape=circle,draw,fill=black] (13) {}
      (2,0.5)node[shape=circle,draw,fill=black] (14) {}

     (1,1.5)node[shape=rectangle,draw,fill=red] (0) {}
     (1,2)node[shape=circle,draw,fill=black] (3) {}
     (1.5,1.5)node[shape=rectangle,draw,fill=black] (5) {}
     (1.5,2)node[shape=circle,draw,fill=black] (4) {}
      (2,1.5)node[shape=circle,draw,fill=black] (6) {}
      (2.5,1.5)node[shape=circle,draw,fill=black] (7) {}

      (1,-.5)node[shape=circle,draw,fill=black] (16) {}
       (1.5,-.5)node[shape=circle,draw,fill=black] (17) {};
     
      \draw[-](10)--(8);
      \draw[-](8)--(9);

      \draw[-](12)--(13);
      \draw[-](12)--(14);

       \draw[-](8)--(12);

      \draw[-](0)--(3);

      \draw[-](0)--(5);
      \draw[-](5)--(4);
      \draw[-](5)--(6);

      \draw[-](9)--(16);
      \draw[-](13)--(17);

     \draw[-](6)--(7);
     \draw[dotted](0)--(8);

     \end{tikzpicture}
}


\newcommand{\seedsetedois}{
\begin{tikzpicture}[scale=.7,auto=left,every node/.style={circle,scale=0.3}]

\path
     (1,.5)node[shape=rectangle,draw,fill=black] (8) {}
     (1,0)node[shape=circle,draw,fill=black] (9) {}
     (1.5,0.5)node[shape=rectangle,draw,fill=black] (12) {}
     (1.5,0)node[shape=circle,draw,fill=black] (13) {}
      (2,0.5)node[shape=circle,draw,fill=black] (14) {}
      (2.5,0.5)node[shape=circle,draw,fill=black] (15) {}

     (1,1.5)node[shape=rectangle,draw,fill=red] (0) {}
     (1,2)node[shape=circle,draw,fill=black] (3) {}
     (1.5,1.5)node[shape=rectangle,draw,fill=black] (5) {}
     (1.5,2)node[shape=circle,draw,fill=black] (4) {}
      (2,1.5)node[shape=circle,draw,fill=black] (6) {}
      (2.5,1.5)node[shape=circle,draw,fill=black] (7) {};

      \draw[-](8)--(9);

      \draw[-](12)--(13);
      \draw[-](12)--(14);
      \draw[-](14)--(15);

       \draw[-](8)--(12);

      \draw[-](0)--(3);

      \draw[-](0)--(5);
      \draw[-](5)--(4);
      \draw[-](5)--(6);

     \draw[-](6)--(7);
     \draw[dotted](0)--(8);

     \end{tikzpicture}
}


\captionsetup[subfigure]{labelformat=empty}

\begin{figure}[H]
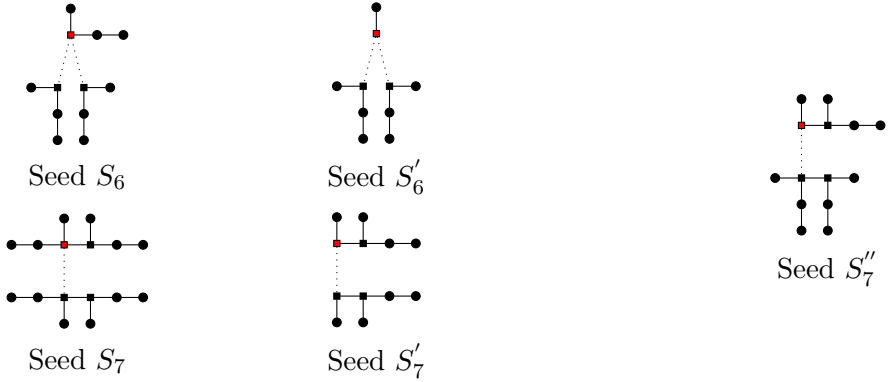

\centering
\begin{minipage}[h]{.5\textwidth}
\centering

 \begin{subfigure}{0.49\textwidth}
     \centering
     \seedseis
     \caption{Seed $S_{6}$}
     \label{fig:s6}
 \end{subfigure}
 \hfill
 \begin{subfigure}{0.49\textwidth}
     \centering
     \seedseisum
     \caption{Seed $S_{6}^{'}$}
     \label{fig:s61}
 \end{subfigure}
 
 \medskip
 \begin{subfigure}{0.49\textwidth}
     \centering
     \seedsete
     \caption{Seed $S_{7}$}
     \label{fig:s7}
 \end{subfigure}
 \hfill
 \begin{subfigure}{0.49\textwidth}
     \centering
     \seedsetedois
     
     \caption{Seed $S_{7}^{'}$}
     \label{fig:s71}
 \end{subfigure}

\end{minipage}\hfill
\begin{minipage}[h]{.5\textwidth}

 \begin{subfigure}{1.05\textwidth}
     \centering
     \seedseteum
     \caption{Seed $S_{7}^{''}$}
\label{fig:s72}
 \end{subfigure}
\end{minipage}
\caption{The seeds of Definition~\ref{def_seeds} for $d\in \{6,7\}$.\label{fig:seeds}}
\end{figure}

It turns out that the entire class of trees that may be generated by unfoldings of seeds in $\{S_d\}_{d \geq 0}$ may also be generated recursively using the operation $\odot$.
\begin{defn}\label{def_T}
Let $\mathcal{T}^\ast$ be the set of trees defined as follows:
\begin{enumerate}
    \item [(i)] $K_1$ is the single tree in $\mathcal{T}^\ast$ with height $0$.
    \item[(ii)] Let $k,p$ be positive integers and consider trees $T_0,T_1,\ldots,T_p \in \mathcal{T}^\ast$, rooted at a main root, with height $k-1$. Then $T=T_0\odot (T_1,\ldots,T_p) \in \mathcal{T}^\ast$.
\end{enumerate}
\end{defn}
Note that, for a tree $T$ defined in (ii), the diameter is $2k-1$ if $p=1$ and the diameter is $2k$ if $p\geq2$. We also observe that $v_0\in T_0$ is the main root of $T= T_0 \odot (T_1,\ldots,T_p)$ and the edge between $T_0$ and $T_1$ is the central edge of $T=T_0\odot T_1$. It is clear that the tree $T$ generated by trees of height $k$ in (ii) has height $k+1$ when it is rooted at a main vertex. 

Recall that, if $S$ is a seed of diameter $d$, $\mathcal{T}(S)$ denotes the set of trees of diameter $d$ that are unfoldings of $S$.
\begin{prop}\label{prop_equivalence}
For any tree $T$ of diameter $d\geq 0$, we have $T \in \mathcal{T}^\ast$ if and only if $T \in \mathcal{T}(S_d)$.
\end{prop}

\begin{proof}
To show that any tree $T \in \mathcal{T}(S_d)$ lies in $\mathcal{T}^\ast$, we prove the following two claims:
\begin{itemize}
    \item[1.] Any seed $S_d$ lies in $\mathcal{T}^\ast$.
    
    \item[2.] Assume that $T \in \mathcal{T}^\ast$. Then any tree $T'$ obtained from $T$ by a CBD lies in $\mathcal{T}^\ast$.
\end{itemize}

For part 1, note that $S_0 \in \mathcal{T}^\ast$ by Definition~\ref{def_T}(i). We have $S_1=S_0 \odot S_0$ and $S_2=S_0 \odot (S_0 \odot S_0)$, and therefore they are elements of height $1$ in $\mathcal{T}^\ast$ by Definition~\ref{def_T}(ii). For larger values of $d$, we proceed inductively. Note that, for all $k \geq 2$, the seed $S_{2k-3}$ (viewed as the rooted tree of Definition~\ref{def_seeds}) has height $k-1$. As a consequence, assuming that $S_{2k-3} \in \mathcal{T}^\ast$, we have $S_{2k-1}=S_{2k-3}\odot S_{2k-3}$ and $S_{2k}=S_{2k-3}\odot (S_{2k-3},S_{2k-3})$ in $\mathcal{T}^\ast$ by Definition~\ref{def_T}(ii).

We now prove part 2. We proceed by induction on $k$. For each $k\in\mathbb{N}$, we show that, for any tree $T \in \mathcal{T}^\ast$ of diameter $2k-1$ or $2k$, any CBD of a branch of $T$ leads to a tree $T'$ in $\mathcal{T}^\ast$.

The base case $k=1$ is trivial, as all trees with diameter at most two lie in $\mathcal{T}^\ast$. Suppose that the statement holds for all trees in $\mathcal{T}^\ast$ with diameter at most $2k$, and fix $T=T_0\odot(T_1,\ldots,T_p) \in \mathcal{T}^\ast$ with diameter $d\in\{2k+1,2k+2\}$. By Definition~\ref{def_T}, each $T_i$ lies in $\mathcal{T}^\ast$ and has height $k$, so that its diameter is at most $2k$. 

Let $T'$ be the tree produced by an $s$-CBD of a branch $U_j$ of $T$ at a vertex $u$.

\vspace{5pt}

\noindent \textit{Case 1.} Assume that $u\neq v_0$, so that $u\in V(T_i)$ for some $i\in\{0,\ldots,p\}$. By the induction hypothesis, the tree $T'_i$ produced by an $s$-CBD of $U_j$ at $u$ lies in $\mathcal{T}^\ast$. By the definition of branch duplication, $T_i$ and $T_i'$ have the same height. We conclude that $T'\in \mathcal{T}^\ast$ because $T'=T_0\odot(T_1,\ldots,T_{i-1},T'_i,T_{i+1},\ldots,T_p)$, if $i\neq 0$, or $T'=T'_0\odot(T_1,\ldots,T_p)$, if $i= 0$.

\vspace{5pt}

\noindent \textit{Case 2.} Assume that $u=v_0$, so that either the chosen branch is equal to $T_i$ for some $i\in\{1,\ldots,p\}$ or the chosen branch is a branch in $T_0$. In the latter case, we simply repeat the argument of case 1 with $T_0'$ being produced by a CBD in $T_0$. Otherwise, $u=v_0$ and $U_j=T_i$, so that $T'\in \mathcal{T}^\ast$ because $T' = T_0\odot(T_1,\ldots,T_i, T^{(1)}_i,\ldots,T^{(s)}_i,T_{i+1},\ldots,T_p)$.

To conclude the proof of Proposition~\ref{prop_equivalence}, we must show that every tree $T \in \mathcal{T}^\ast$ with diameter $d$ lies in $\mathcal{T}(S_d)$. This will again be done by induction on $k$, the height of the tree $T \in \mathcal{T}^\ast$ (viewed as a rooted tree with root at a main vertex).

For $k=1$, the statement is trivially true, because $S_1$ and $S_2$ are the only seeds with diameter $2k-1=1$ and $2k=2$, respectively. Suppose that for some $k\in\mathbb{N}$ every $T\in \mathcal{T}^\ast$ of diameter $2k-1$ is an unfolding of $S_{2k-1}$ and every $T\in T^\ast$ of diameter $2k$ is an unfolding of $S_{2k}$. 

For the induction step, first fix $T\in \mathcal{T}^\ast$ of diameter $2k+1$. Then, $T=T_0\odot T_1$, for some $T_0,T_1 \in \mathcal{T}^\ast$ of height $k$. By hypothesis, $T_0$ and $T_1$ may be folded until we arrive at  their respective seeds $S^{(0)}$ and $S^{(1)}$, respectively. There are three possibilities:
\begin{enumerate}
    \item [(i)] $S^{(0)}=S_{2k-1}$ and $S^{(1)}=S_{2k-1}$. In this case, $S^{(0)}\odot S^{(1)}=S_{2k-1}\odot S_{2k-1}=S_{2k+1}$ is a folding of $T$, as required.
    \item [(ii)] $S^{(0)}=S_{2k-1}$ and $S^{(1)}=S_{2k}$ (the case $S^{(0)}=S_{2k}$ and $S^{(1)}=S_{2k-1}$ is analogous). In this case, 
    \begin{eqnarray*}
    S^{(0)}\odot S^{(1)}&=&S_{2k-1}\odot S_{2k}\\
    &=& S_{2k-1} \odot \left(S_{2k-3}\odot(S_{2k-3},S_{2k-3}) \right).
    \end{eqnarray*}
    The pair $(S_{2k-3},S_{2k-3})$ may be folded into a single occurrence of $S_{2k-3}$ without decreasing the diameter, so that we get 
    $$S_{2k-1} \odot \left(S_{2k-3} \odot S_{2k-3}\right) = S_{2k-1} \odot S_{2k-1}=S_{2k+1},$$ as required.
    
    \item [(iii)] $S^{(0)}=S_{2k}$ and $S^{(1)}=S_{2k}$. This case is similar to case (ii), as \begin{eqnarray*}
    S^{(0)}\odot S^{(1)}&=&S_{2k}\odot S_{2k}\\
    &=& \left(S_{2k-3}\odot(S_{2k-3},S_{2k-3})\right) \odot \left(S_{2k-3}\odot(S_{2k-3},S_{2k-3}) \right).
    \end{eqnarray*}
    In this case, we can fold each pair $(S_{2k-3},S_{2k-3})$ into a single occurrence of $S_{2k-3}$, and the result follows as above.
\end{enumerate}

To conclude the proof, assume that $T\in \mathcal{T}^\ast$ has diameter $2k+2$. Then, $T=T_0\odot (T_1,\ldots,T_p)$, $p\geq 2$, for some $T_0,T_1,\ldots,T_p \in \mathcal{T}^\ast$ of height $k$. Each $T_i$ may be folded down to $S_{2k-1}$ or to $S_{2k}$, according to its diameter. Each occurrence of $S_{2k}$ may be replaced by $S_{2k-3}\odot(S_{2k-3},S_{2k-3})$, which can be folded to $S_{2k-3} \odot S_{2k-3}=S_{2k-1}$. This means that we reach
$S_{2k-1} \odot (S_{2k-1},\cdots,S_{2k-1})$, where the vector contains at least two terms. If it has more than two terms, additional terms may be removed by foldings of branches $S_{2k-1}$ without decreasing the diameter. When we reach $S_{2k-1} \odot (S_{2k-1},S_{2k-1})$, no further folding can be performed, as it would decrease the diameter. The result follows because $S_{2k-1} \odot (S_{2k-1},S_{2k-1})=S_{2k+2}$. 
\end{proof}

The trees generated by unfoldings of the other seeds in~Definition~\ref{def_seeds} may also be described by decompositions involving the operation $\odot$, as described in the proposition below. The arguments in the proof are quite similar to the ones used to prove Proposition~\ref{prop_equivalence} and is therefore omitted. The interested reader finds the proof of item (ii) in the appendix.  
\begin{prop}\label{equivalence_other_seeds}
    Let $T$ be a tree and $k\geq1$. The following hold:
    \begin{enumerate}
        \item [(i)] $T\in\mathcal{T}(S'_{2k+2})$ if, and only if, there exist $T_1,\ldots,T_p\in \mathcal{T}^\ast, p\geq 2$, 
        of height $k$ and $T_0\in \mathcal{T}^\ast$ of height $k-1$ such that $T=T_0\odot(T_1,\ldots,T_p)$;
        \item [(ii)] $T\in\mathcal{T}(S'_{2k+3})$ if, and only if, there exist $T_1,\ldots,T_p,T'_1,\ldots,T'_q\in \mathcal{T}^\ast, p,q\geq 1$, of height $k$ and $T_0,T'_0\in \mathcal{T}^\ast$ of height $k-1$ such that $T=(T_0\odot(T_1,\ldots,T_p))\odot (T'_0\odot(T'_1,\ldots,T'_q))$;
        \item [(iii)] $T\in\mathcal{T}(S''_{2k+3})$ if, and only if, there exist $T_1,\ldots,T_p,T'_0,\ldots,T'_q\in \mathcal{T}^\ast, p,q\geq 1$, of height $k$ and $T_0\in \mathcal{T}^\ast$ of height $k-1$ such that $T=(T_0\odot(T_1,\ldots,T_p))\odot (T'_0\odot(T'_1,\ldots,T'_q))$.
    \end{enumerate}
    \end{prop}

To conclude this section, we present a useful connection between a symmetric matrix $M$ with underlying tree $T = T_0 \odot (T_1,\ldots,T_p)$ and induced submatrices corresponding to the subtrees $T_i$.
\begin{lemat}\label{lema_define_max}
Let $T_0,\ldots,T_p$ be rooted trees with roots $v_0,\ldots,v_p$, respectively, where $p\geq 1$. Let $T = T_0 \odot (T_1,\ldots,T_p)$. Given $M_i \in \mathcal{S}(T_i)$, for $i\in\{0,\ldots,p\}$ and $\delta \neq 0$, let $M$ be the matrix $M=(m_{ij})\in \mathcal{S}(T)$ where $M[T_i]=M_i$ and $m_{v_0v_i}=\delta$ for all $i \in \{1,\ldots,p\}$ (see Figure~\ref{lemma_2.10fig}). The following hold:
\begin{enumerate}
    \item [(i)] $\lambda_{\min}(M)<\lambda<\lambda_{\max}(M)$, for all $\lambda\in$ $\bigcup_{\ell=0}^p \Spec(M_\ell)$.
    \item [(ii)] Given $y>\lambda$, for all $\lambda\in \bigcup_{\ell=0}^p \Spec(M_\ell)$, there exists $\delta(y)>0$ such that $\lambda_{\max}(M) = y$.
    \item [(iii)] Given $y<\lambda$, for all $\lambda\in \bigcup_{\ell=0}^p \Spec(M_\ell)$, there exists $\delta(y)>0$ such that $\lambda_{\min}(M) = y$.
\end{enumerate}
\end{lemat}

\begin{proof}
Let $p\geq 1$ and let $T_0,\ldots,T_p$ be rooted trees with roots $v_0,\ldots,v_p$, respectively for a given $p\in\mathbb{N}$. Given $M_i \in \mathcal{S}(T_i)$ and $\delta \neq 0$, define $M\in \mathcal{S}(T)$, where  $T = T_0 \odot (T_1,\ldots,T_p)$, such that $M[T_i]=M_i$ and $m_{v_0v_i}=\delta$ for all $i \in \{1,\ldots,p\}$.

\begin{figure}
$$
\begin{pNiceMatrix}[first-row][first-col]
               & V(T_0)   & V(T_1) & V(T_2) & \cdots       & V(T_p)   \\[0.2cm]
        V(T_0) & M_0      & A_{01} & A_{02} & \cdots & A_{0p}   \\[0.3cm]
        V(T_1) & A^T_{01} & M_1    & \bf{0}      & \cdots & \bf{0}        \\[0.2cm]
        V(T_2) & A^T_{02} &   \bf{0}    & M_2    & \ddots & \vdots   \\[0.2cm]
        \vdots~~       & \vdots   & \vdots & \ddots & \ddots & \bf{0}        \\[0.2cm]
        V(T_p) & A^T_{0p} & \bf{0}      & \bf{0}      & \dots  & M_p      \\
    \end{pNiceMatrix} \text{, where }
A_{0i} = \begin{pNiceMatrix}[first-row][first-col]
            & v_i    &        &        &     \\
        v_0 & \delta & 0      & \cdots & 0   \\
            &   0    & 0      & \cdots & 0   \\
            & \vdots & \vdots & \ddots & 0   \\
            &   0    & 0      & 0      & 0   \\
    \end{pNiceMatrix}
$$
\caption{Matrix $M$ as in the statement of Lemma~\ref{lema_define_max}. The rows and columns of the matrix are ordered according to the tree $T_i$ they come from. \label{lemma_2.10fig}}
\end{figure}

Consider $\beta=\max\{\lambda\in\mbox{Spec}(M_{i}): 0\leq i \leq p\}$, so that $\beta=\lambda_{\max}(M_{\ell})$ for some $0\leq \ell \leq p$. 

We start with part (i). First, assume that $\ell> 0$. Since $\beta=\lambda_{\max}(M_\ell)$, Theorem~\ref{thm:simpleroots} tells us that an application of \texttt{Diagonalize}$(M[T_\ell],-\beta)$ with root $v_\ell$ assigns negative values to all vertices of $T_\ell$ except $v_\ell$, for which the value is 0. This coincides with the values assigned to these vertices in an application of \texttt{Diagonalize}$(M,-\beta)$ with root $v_0$ before processing $v_0$. Since $v_{\ell}$ is a child of $v_{0}$ with value $0$, when the algorithm processes $v_{0}$, it redefines $d_{v_j}=2>0$ and $d_{v_0}<0$ for some child $v_j$ of $v_0$ (possibly $j=\ell$). Therefore, according to Theorem~\ref{inertia}(a), $\beta<\lambda_{\max}(M)$.

Next assume that $\beta=\lambda_{\max}(M_0)>\lambda_{\max}(M_i)$ for all $i>0$. As in the previous case, an application of \texttt{Diagonalize}$(M[T_0],-\beta)$ with root $v_0$ assigns negative values to all vertices of $T_0$ except $v_0$. Moreover, by Theorem~\ref{inertia}(b), applying \texttt{Diagonalize}$(M[T_i],-\beta)$ to each $T_i$ with root $v_i$ assigns negative values to all vertices of $T_i$. As before, all these values coincide with the values assigned by an application of \texttt{Diagonalize}$(M,-\beta)$ with root $v_0$.

When \texttt{Diagonalize}$(M,-\beta)$ processes $v_0$, it assigns the value
\begin{equation}\label{eq1}
    d_{v_{0}}=(m_{v_0v_0}-\beta)-\sum_{w\in C_{T_0}(v_0)}\frac{m_{v_{0}w}^{2}}{d_w}-\sum_{i=1}^{p}\frac{m_{v_{0}v_{i}}^{2}}{d_{v_{i}}},
\end{equation}
where $C_{T_0}(v_0)$ denotes the neighborhood of $v_0$ in $T_0$. Also note that, when we run \texttt{Diagonalize}$(M[T_0],-\beta)$ with root $v_0$, we obtain the final permanent value $$0=(m_{v_0v_0}-\beta)-\sum_{w\in C_{T_0}(v_0)}\frac{m_{v_{0}w}^{2}}{d_w}.$$

Thus, as $d_{v_{i}}<0$ for $1\leq i \leq p$, equation \eqref{eq1} becomes  $$d_{v_{0}}=-\sum_{i=1}^{p}\frac{m_{v_{0}v_{i}}^{2}}{d_{v_{i}}}>0,$$
so that $\beta<\lambda_{\max}(M)$ by Theorem~\ref{inertia}(a). 

To prove that $\lambda_{\min}(M)<\lambda$, for all $\lambda\in \bigcup_{\ell=0}^p \Spec(M_\ell)$, it suffices to apply this result to $-M$, as $\lambda_{\min}(M)=-\lambda_{\max}(-M)$.

To prove part (ii), fix $y>\beta$. We run \texttt{Diagonalize}$(M,-y)$ with root $v_0$. Just before we process $v_{0}$, all its children have been assigned negative values. Then we have
\begin{equation}\label{eq2}
d_{v_{0}}=(m_{v_0v_0}-y)-\sum_{w\in C_{T_0}(v_0)}\frac{m_{v_{0}w}^{2}}{d_w}-\sum_{i=1}^{p}\frac{\delta^{2}}{d_{v_{i}}}.
\end{equation}
Moreover, when we run \texttt{Diagonalize}$(M[T_0],-y)$, it assigns final permanent value $$0>d^{(T_0)}_{v_0}=(m_{v_0v_0}-y)-\sum_{w\in C_{T_0}(v_0)}\frac{m_{v_{0}w}^{2}}{d_w},$$
since $y>\beta\geq\lambda_{\max}(M_{0})$.

To obtain $d_{v_0}=0$ in~\eqref{eq2} we can set 
$$\delta(y)=\sqrt{\frac{\left((m_{v_0v_0}-y)-\sum_{w\in C_{T_0}(v_0)}\frac{m_{v_{0}w}^{2}}{d_w}\right)}{\sum_{i=1}^{p}\frac{1}{d_{v_{i}}}}}.$$
The expression within the square root is positive because $d_{v_{i}}<0$ for $1\leq i \leq p$.

Item (iii) may be derived from item (ii) by considering the matrix $-M$.
\end{proof}

\section{Strongly realizable sets}\label{sec:strongly_realizable}

In this section, we shall state the technical result that implies the validity of Theorem~\ref{thm:main}, namely Theorem~\ref{main_th} below. This technical result allows us to inductively define a set of real numbers (of size $d+1$) that is equal to the distinct eigenvalues in the spectrum of a symmetric matrix $M(T)$ whose underlying graph is a tree $T \in \mathcal{T}^\ast$ with diameter $d$.

\begin{defn}
Let $T=(V,E)$ be a tree with main root $v$. Let $M\in \mathcal{S}(T)$. For each $\lambda\in\DSpec(M)$ we define $$L(M,\lambda) = \min_{u\in V(T)}\{d(v,u):\tilde{d}_u = 0 \},$$ where $\tilde{d}_u$ denotes the final value assigned to $u$ by \texttt{Diagonalize}$(M,-\lambda)$ with root $v$.
\end{defn}

In the definition below, and in the remainder of the paper, we shall use the notation $\{\lambda_1<\cdots<\lambda_{\ell}\}$ to refer to a set  $\{\lambda_1,\ldots,\lambda_{\ell}\}$ of real numbers such that $\lambda_1<\cdots<\lambda_{\ell}$.

\begin{defn}
Given $k\in\mathbb{N}$, a set of real numbers $A = \{\lambda_0<\cdots<\lambda_{2k}\}$  is said to be \emph{strongly realizable} in a family of rooted trees $\mathcal{C}=\{T_i\}_{i\in\mathcal{I}}$, where $\mathcal{I}$ is a set of indices, if the following holds for any $T\in \mathcal{C}$ of height $k$ and root $v$. There exists $M\in \mathcal{S}(T)$ satisfying:
\begin{enumerate}
    \item $\DSpec(M) \subseteq A$;
    \item $L(M,\lambda_{2i})=0$, $0\leq i\leq k$;
    \item $m_{M[T-v]}(\lambda_{2i-1})=m_{M}(\lambda_{2i-1})+1$, $1\leq i\leq k$.
\end{enumerate}
A matrix $M$ with the above properties is said to be a \emph{strong realization} of $A$ in $\mathcal{C}$.
\end{defn}
Note that, by this definition, the values $\lambda_0,\lambda_2,\ldots,\lambda_{2k}$ must be in the spectrum of $M$. 

\begin{example}
We show that the set $\{\lambda_0,\ldots,\lambda_4\}=\{-2,-1,0,1,3\}$ is strongly realizable in $\mathcal{T}^\ast$. To this end, we need to show that the following holds for any $T\in \mathcal{T}^\ast$ with diameter $d\in \{3,4\}$. If $d=3$, there must be a matrix $M$ with underlying tree $T$ whose spectrum contains $-2,0,3$ and at least one of the elements $-1$ and $1$. For $d=4$, the set of distinct eigenvalues must be equal to $\{-2,-1,0,1,3\}$. Moreover, conditions (2) and (3) must hold in both cases. Weights that satisfy these properties are given in Figure~\ref{ex:realizable}. Note that the diameter is equal to $3$ if $p=1$ and equal to $4$ if $p\geq 2$.
The properties (1)-(3) may be easily checked by applying \texttt{Diagonalize}$(M,-\lambda)$ for values of $\lambda$ in this set. We may further verify that $m_{M}(-2)=1$, $m_{M}(-1)=p-1$, $m_{M}(0)=1-p+\sum_{i=1}^p t_i$, $m_{M}(1)=t_0+p-1$ and $m_{M}(3)=1$, so that the multiplicities add up to $|V(T)|$. Observe that $\lambda=-1$ is an eigenvalue of $M$ if and only if the diameter is 4.
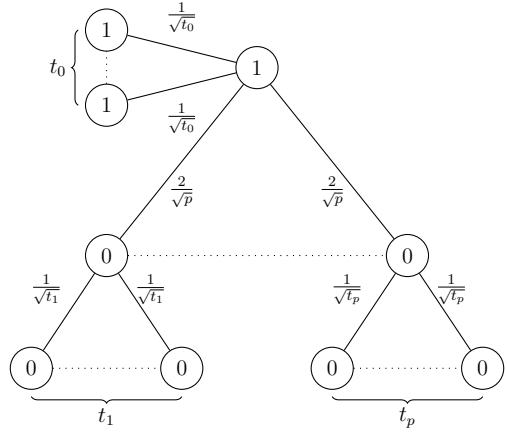
\begin{figure}
\centering
\begin{tikzpicture}
[scale=1,auto=left,every node/.style={circle,scale=0.7}]

\path( 0,0)  node[shape=circle,draw,minimum size=0.8cm,label=above:,inner sep=0] (v1)  {$0$}
     ( 2,0)  node[shape=circle,draw,minimum size=0.8cm,label=above:,inner sep=0] (v2)  {$0$}
     ( 1,1.5)node[shape=circle,draw,minimum size=0.8cm,label=above:,inner sep=0] (v12) {$0$}

     ( 4,0)  node[shape=circle,draw,minimum size=0.8cm,label=above:,inner sep=0] (v3)  {$0$}
     ( 6,0)  node[shape=circle,draw,minimum size=0.8cm,label=above:,inner sep=0] (v4)  {$0$}
     ( 5,1.5)node[shape=circle,draw,minimum size=0.8cm,label=above:,inner sep=0] (v11) {$0$}

     ( 3,4)  node[shape=circle,draw,minimum size=0.8cm,label=above:,inner sep=0] (v13) {$1$}

     (1,3.5) node[shape=circle,draw,minimum size=0.8cm,label=above:,inner sep=0] (v17) {$1$}
     (1,4.5) node[shape=circle,draw,minimum size=0.8cm,label=above:,inner sep=0] (v18) {$1$};

\draw[dotted](v17) -- (v18) ; 
\draw[-,below] (v13) edge node{$\frac{1}{\sqrt{t_0}}$} (v17) ; 
\draw[-,above] (v13) edge node{$\frac{1}{\sqrt{t_0}}$} (v18) ; 

\draw[decoration={brace,raise=5pt},decorate]
  (0.8,3.5) -- node[right=0pt,left,xshift=-0.5em] {$t_0$} (0.8,4.5);
  
\draw[dotted](v1) -- (v2) ; 
\draw[-](v1) edge node{$\frac{1}{\sqrt{t_1}}$} (v12) ; 
\draw[-,pos=0.375,above](v2) edge node{$\frac{1}{\sqrt{t_1}}$} (v12) ;

\draw[decoration={brace,mirror,raise=5pt},decorate]
  (0,-0.2) -- node[right=0pt,below,yshift=-0.5em] {$t_1$} (2,-0.2);
  
\draw[dotted](v3) -- (v4) ; 
\draw[-](v3) edge node{$\frac{1}{\sqrt{t_p}}$}  (v11) ; 
\draw[-,pos=0.375,above](v4) edge node{$\frac{1}{\sqrt{t_p}}$}  (v11) ; 

\draw[decoration={brace,mirror,raise=5pt},decorate]
  (4,-0.2) -- node[right=0pt,below,yshift=-0.5em] {$t_p$} (6,-0.2);

\draw[dotted](v12) -- (v11) ; 
\draw[-,below](v13) edge node{$\frac{2}{\sqrt{p}}$}  (v12) ; 
\draw[-,below](v13) edge node{$\frac{2}{\sqrt{p}}$}  (v11) ;

\end{tikzpicture}
\caption{A weighted tree of diameter $4$ with spectrum $\{-2,-1,0,1,3\}$ that satisfies (2) and (3). Note that choosing $p\geq 2$ and $t_i \geq 1$ for all $i\in \{0,\ldots,p\}$ produces all possible trees in $\mathcal{T}^\ast$.  \label{ex:realizable}
}
\end{figure}

\end{example}

The main technical result in this section is the following. It states that, for every $k \in \mathbb{N}$, there exists a set of real numbers $C_k=\{\lambda_0<\lambda_1<\cdots<\lambda_{2k+1}\}$ such that the subsets $C_k^{(0)}=\{\lambda_0<\lambda_1<\cdots<\lambda_{2k}\}$ and $C_k^{(1)}=\{\lambda_1<\lambda_2<\cdots<\lambda_{2k+1}\}$ are both strongly realizable in $\mathcal{T}^\ast$. Moreover, as long as $\theta,\delta>0$ are sufficiently small, the sets $\{\lambda_0+\theta<\lambda_1<\cdots<\lambda_{2k}\}$ and $\{\lambda_1<\cdots<\lambda_{2k}<\lambda_{2k+1}+\delta\}$ must also be strongly realizable in $\mathcal{T}^\ast$.
\begin{theorem}\label{main_th}
Let $\alpha<\beta$ be real numbers. For every $k\in\mathbb{N}$, there exists a set of real numbers $C_k=\{\lambda_0<\lambda_1<\cdots<\lambda_{2k+1}\}$, where $\lambda_k=\alpha$ and $\lambda_{k+1}=\beta$, such that the following holds for every $T \in \mathcal{T}^\ast$ with height $k$, diameter $d$ and main root $v$. There exist matrices $M_1^{(k)},M_2^{(k)}\in \mathcal{S}(T)$ satisfying the following:
\begin{itemize}
    \item[(i)] $\DSpec(M_1^{(k)}) = \{\lambda_0,\ldots,\lambda_{2k}\}$, if $d=2k$; $\DSpec(M_1^{(k)}) = \{\lambda_0,\ldots,\lambda_{2k}\}\setminus\{\lambda_1\}$, if $d=2k-1$;
    \item[(ii)]  $\DSpec(M_2^{(k)})=\{\lambda_1,\ldots,\lambda_{2k+1}\}$, if $d=2k$; $\DSpec(M_2^{(k)})=\{\lambda_1,\ldots,\lambda_{2k+1}\}\setminus\{\lambda_{2k}\}$, if $d=2k-1$;
    \item[(iii)] $L(M_1^{(k)},\lambda_{2i})=0=L(M_2^{(k)},\lambda_{2i+1})$, for $i\in\{0,\ldots,k\}$;
    \item[(iv)] $m_{M_1^{(k)}[T-v]}(\lambda_{2i-1})=m_{M_1^{(k)}}(\lambda_{2i-1})+1$ and $m_{M_2^{(k)}[T-v]}(\lambda_{2i})=m_{M_2^{(k)}}(\lambda_{2i})+1$, for $i\in\{1,\ldots,k\}$.
\end{itemize}
Moreover, the following are satisfied.
\begin{itemize}
    \item[(v)] Let $y_k=\frac{\beta-\alpha}{2^{k-1}}$. For all $\theta\in(0,y_k)$, there exists $M_{1,\theta}^{(k)}$ such that  $$\DSpec(M_{1,\theta}^{(k)})\subseteq\{\lambda_0+\theta,\lambda_1,\ldots,\lambda_{2k}\},$$ $L(M_{1,\theta}^{(k)},\lambda_{0}+\theta)=0$, and, for all $i\in\{1,\ldots,k\}$, we have $L(M_{1,\theta}^{(k)},\lambda_{2i})=0$ and $m_{M_{1,\theta}^{(k)}[T-v]}(\lambda_{2i-1})=m_{M_{1,\theta}^{(k)}}(\lambda_{2i-1})+1$.
    \item[(vi)] For all $\delta\in(0,y_k)$, there exists $M_{2,\delta}^{(k)}$ such that $$\DSpec(M_{2,\delta}^{(k)})\subseteq\{\lambda_1,\ldots,\lambda_{2k},\lambda_{2k+1}+\delta\},$$ $L(M_{2,\delta}^{(k)},\lambda_{2k+1}+\delta)=0$, and, for all $i\in\{1,\ldots,k\}$, we have $L(M_{2,\delta}^{(k)},\lambda_{2i-1})=0$ and $m_{M_{2,\delta}^{(k)}[T-v]}(\lambda_{2i})=m_{M_{2,\delta}^{(k)}}(\lambda_{2i})+1$.
\end{itemize}
\end{theorem}
We emphasize that, in our proof of Theorem~\ref{main_th}, the set $C_k$ does depend on $k$, in the sense that $C_{k+1}$ is not obtained from $C_k$ by the inclusion of two new elements. The proof of Theorem~\ref{main_th} will be the subject of the next section. We now observe that it immediately implies that Theorem~\ref{thm:main} holds for trees in $\mathcal{T}^\ast$. 

\begin{proof}[Proof of Theorem~\ref{thm:main} for trees in $\mathcal{T}^\ast$]
Let $T\in \mathcal{T}^\ast$ with diameter $d$.
Theorem~\ref{main_th}(i) tells us that it admits a matrix $M(T)$ with $d+1$ distinct eigenvalues that is a realization of a set of $d+1$ real numbers ($C_k\setminus\{\lambda_{2k+1}\}$, if $d=2k$; $C_k\setminus \{\lambda_1,\lambda_{2k+1}\}$, if $d=2k-1$). By Theorem~\ref{thm:LB}, we deduce that $q(T)=d+1$.
\end{proof}

\section{Proof of Theorem~\ref{main_th}}\label{sec:proof_technical}
Theorem~\ref{main_th} will be proved by induction. One of the main ingredients for the step of induction is the following result, which gives a construction that allows us to extend the spectra of a set of matrices to a larger matrix in terms of the operation $\odot$. 

\begin{lemat}\label{multiplicidades}
Let $\mathcal{C}_1$ and $\mathcal{C}_2$ be two families of rooted trees. Let $k_1$ and $k_2$ be nonnegative integers and assume that $A_1=\{\lambda_0<\cdots<\lambda_{2k_1}\}$ and  $A_2=\{\mu_0<\cdots<\mu_{2k_2}\}$ are two strongly realizable sets in $\mathcal{C}_1$ and $\mathcal{C}_2$, respectively, such that $(A_1\cup A_2)\setminus(A_1\cap A_2)=\{a, b\}$, with $a =\min\{\lambda_0,\mu_0\}$ and  $b=\max\{\lambda_{2k_1},\mu_{2k_2}\}$. 
Suppose that there is a partition $A_1 \cap A_2 = \Lambda_1 \cup \Lambda_2$ with the following property. For any trees $T_1\in\mathcal{C}_1$ and $T_2\in\mathcal{C}_2$ with height $k_1$ and $k_2$ and root $v_1$ and $v_2$, respectively, assume that there exist a strong realization $M_1(T_1)$ of $A_1$ and a strong realization $M_2(T_2)$ of $A_2$ such that
\begin{enumerate}
    \item[(i)] For all $\lambda\in \Lambda_1$, we have $L(M_1(T_1),\lambda)=0$ and $m_{M_2[T_2-v_2]}(\lambda)=m_{M_2(T_2)}(\lambda)+1$.
    \item[(ii)] For all $\lambda\in \Lambda_2$, we have $L(M_2(T_2),\lambda)=0$ and $m_{M_1[T_1-v_1]}(\lambda)=m_{M_1(T_1)}(\lambda)+1$.
\end{enumerate}
Then the following holds for a tree $T= T_0 \odot (T_1,\ldots,T_p)$ with main root $v_0$, where $T_1,\ldots,T_p\in\mathcal{C}_1$, $p\geq1$, have height $k_1$, and $T_0\in\mathcal{C}_2$ has height $k_2$. Consider a matrix $M\in \mathcal{S}(T)$ for which $M[T_0]=M_2(T_0)$ and $M[T_i]=M_1(T_i),1\leq i\leq p$ (see Figure~\ref{lemma_3.4fig}). Then there exist $\lambda_{\min},\lambda_{\max}\in \mathbb{R}$ such that the following hold: 
\begin{itemize}
\item[(a)] $\displaystyle{\DSpec(M) =
\begin{cases}
(A_1\cap A_2) \cup \{\lambda_{\min},a,b,\lambda_{\max}\}, & \textrm{ if }p>1 \textrm{ and }a,b\in A_1 ,\\
(A_1\cap A_2) \cup \{\lambda_{\min},a,\lambda_{\max}\}, & \textrm{ if }p>1 \textrm{ and }\{a\}= A_1\cap\{a,b\} ,\\
(A_1\cap A_2) \cup \{\lambda_{\min},b,\lambda_{\max}\}, & \textrm{ if }p>1 \textrm{ and }\{b\}= A_1\cap\{a,b\} ,\\
(A_1\cap A_2) \cup \{\lambda_{\min},\lambda_{\max}\}, & \textrm{ if }p=1 \textrm{ or } (p\geq1 \textrm{ and }a,b\in A_2);
\end{cases}}$
\item[(b)] For $\lambda\in A_1\cap A_2$, $$m_M(\lambda)=m_{M_2(T_0)}(\lambda)+\sum_{i=1}^pm_{M_1(T_i)}(\lambda);$$ 
\item[(c)] $L(M,\lambda)=0$ for all $\lambda\in \Lambda_2$;  
\item[(d)] $m_{M[T-v_0]}(\lambda)=m_{M}(\lambda)+1$, for all $\lambda\in \Lambda_1$;
\item[(e)] For $x\in \{a,b\}$, $$m_M(x)=
\begin{cases}
p-1 \textrm{ and } m_{M[T-v_0]}(x)=m_{M}(x)+1, \textrm{ if } x\in A_1\\
0, \textrm{ if } x\in A_2;
\end{cases}$$
\item[(f)] $\lambda_{\max}+\lambda_{\min}=a+b$.
\end{itemize}
\end{lemat}

\begin{proof}

Let $\mathcal{C}_1$ and $\mathcal{C}_2$ be two families of rooted trees. Fix $k_1, k_2, A_1, A_2$, and $\Lambda_1,\Lambda_2$ satisfying the conditions of the lemma.

Let $T = T_0 \odot (T_1,\ldots,T_p)$, where $T_1,\ldots,T_p\in\mathcal{C}_1$, $p\geq1$, have height $k_1$, and $T_0\in\mathcal{C}_2$ has height $k_2$. Let $v_0,\ldots,v_p$ be the root of $T_0,\ldots,T_p$, respectively, and let $w_1,\ldots,w_q$ be the children of $v_0$ in $T_0$. 

Let $M$ be a matrix as defined in the statement of the lemma, depicted in Figure~\ref{lemma_3.4fig}. Note that the entries associated with edges of the form $v_0v_i$, where $i>0$, have not been assigned any particular values.
\begin{figure}
$$
\begin{pNiceMatrix}[first-row][first-col]
               & V(T_0)   & V(T_1) & V(T_2) &        & V(T_p)   \\[0.2cm]
        V(T_0) & M_2(T_0) & A_{01} & A_{02} & \cdots & A_{0p}   \\[0.3cm]
        V(T_1) & A^T_{01} & M_1(T_1)    & 0      & \cdots & 0        \\[0.2cm]
        V(T_2) & A^T_{02} &   0    & M_1(T_2)    & \ddots & \vdots   \\[0.2cm]
               & \vdots   & \vdots & \ddots & \ddots & 0        \\[0.2cm]
        V(T_p) & A^T_{0p} & 0      & 0      & \dots  & M_1(T_p)      \\
    \end{pNiceMatrix} \text{, where }
A_{0i} = \begin{pNiceMatrix}[first-row][first-col]
            & v_i    &        &        &     \\
        v_0 & \ast & 0      & \cdots & 0   \\
            &   0    & 0      & \cdots & 0   \\
            & \vdots & \vdots & \ddots & 0   \\
            &   0    & 0      & 0      & 0   \\
    \end{pNiceMatrix}
$$
\caption{The matrix $M$ given in the statement of Lemma~\ref{multiplicidades}. The rows and columns of the matrix are ordered according to the tree $T_i$ they come from.}
\label{lemma_3.4fig}
\end{figure}

Let $n=|V(T)|$. Clearly,
\begin{eqnarray}
n&=& \sum_{i=0}^p |V(T_i)|\nonumber \\ &=&\sum_{j=0}^{2k_2}m_{M_2(T_0)}(\mu_j) + \sum_{i=1}^p \sum_{j=0}^{2k_1}m_{M_1(T_i)}(\lambda_j) \nonumber\\
&\stackrel{(*)}{=}&p^{\delta_{a1}}+p^{\delta_{b1}} +\sum_{\lambda\in A_1\cap A_2} \left(m_{M_2(T_0)}(\lambda)+ \sum_{i=1}^p m_{M_1(T_i)}(\lambda)\right).  
\label{eq:multiplicities1}
\end{eqnarray}
In (*), $\delta_{a1}= 1$ if $a\in A_1$ and $\delta_{a1}=0$ otherwise, while $\delta_{b1}= 1$ if $b\in A_1$, $\delta_{b1}=0$ otherwise. The term $p^{\delta_{a1}}+p^{\delta_{b1}}$ comes from the multiplicity of $a$ and $b$ as eigenvalues of $M$, which is equal to one for the corresponding trees because the least and the greatest eigenvalues have multiplicity 1 by Theorem~\ref{thm:simpleroots}.

We use the algorithm of Section~\ref{sec:eigenvalue_location} to compute the spectrum of $M$. First, we prove parts (b), (c) and (d) for elements $\lambda \in A_1 \cap A_2 = \Lambda_1\cup \Lambda_2$. Consider an application of \texttt{Diagonalize}$(M,-\lambda)$ with root $v_0$. Before $v_0$ is processed, everything happens as if we had processed \texttt{Diagonalize}$(M_2(T_0),-\lambda)$ and \texttt{Diagonalize}$(M_1(T_i),-\lambda)$, for $i \in \{1,\ldots,p\}$. When we process the main root $v_0$ we have two cases according to whether $\lambda \in \Lambda_1$ or $\lambda \in \Lambda_2$.

If $\lambda\in\Lambda_2$, we have $m_{M_1[T_i-v_i]}(\lambda)=m_{M_1(T_i)}(\lambda)+1$, and $L(M_2(T_0),\lambda)=0$. By Lemma~\ref{proposition}, each $v_i, i\in\{1,\ldots,p\}$ has a child $u_i$ for which \texttt{Diagonalize}$(M,-\lambda)$ assigns $d_{u_i}=0$ (before processing $v_i$). Then, when $v_i$ is processed, it is assigned a negative value, one of its children with value $0$ (possibly $u_i$) is assigned value $2$, and the edge connecting $v_i$ to $v_0$ is deleted. So, processing $v_0$ in \texttt{Diagonalize}$(M,-\lambda)$ is the same as processing $v_0$ in \texttt{Diagonalize}$(M_2,-\lambda)$. In particular, $d_{v_0}=0$, since  $L(M_2(T_0),\lambda)=0$ by hypothesis. Combining these arguments, we see that the multiplicity of $\lambda$ as an eigenvalue of $M$ satisfies $$m_M(\lambda)=m_{M_2(T_0)}(\lambda)+\sum_{i=1}^pm_{M_1(T_i)}(\lambda).$$
We have seen that $L(M,\lambda)=0$. 
    
Next suppose $\lambda\in\Lambda_1$, so that  $L(M_1(T_i),\lambda)=0$, for all $i\in\{1,\ldots,p\}$, and $m_{M_2[T_0-v_0]}(\lambda)=m_{M_2}(\lambda)+1$. In this case, if we consider \texttt{Diagonalize}$(M,-\lambda)$ just before it processes the root $v_0$, we have $d_{v_i}=0$ for all $i\in\{1,\ldots,p\}$, and by Lemma~\ref{proposition} there is $s\in\{1,\ldots,p\}$, such that the algorithm assigns $d_{w_{s}}=0$ before processing $v_0$. Then, when we process $v_0$, we may suppose that the algorithm assigns $d_{w_s}=2$ and $d_{v_0}<0$, and that all of the remaining children with value $0$ are not modified. This also implies that $$m_M(\lambda)=m_{M_2(T_0)}(\lambda)+\sum_{i=1}^pm_{M_1(T_i)}(\lambda).$$ 
Moreover, it is clear that $L(M,\lambda)=1$ and that $m_{M[T-v_0]}(\lambda)=m_{M}(\lambda)+1$.

Next we prove (e). First assume that $x = a\in A_1$. Since it is the least eigenvalue of $M_1(T_i)$ for all $i$, when applying 
\texttt{Diagonalize}$(M_1(T_i),-x)$, we get $d_{v_1}=\cdots=d_{v_p}=0$, while all other vertices in these trees are assigned a positive value (see Theorem~\ref{thm:simpleroots}). Also, since $x<\lambda_{\min}(M_2(T_0))$, \texttt{Diagonalize}$(M_2(T_0),-x)$ assigns positive values to all entries. After processing $v_0$, one of the value $d_{v_i}$ above becomes 2, while $v_0$ is assigned a negative value. By Theorem~\ref{inertia}, this means that $m_M(x)=p-1$ and that there is a single eigenvalue less than it. In particular, if $p=1$, $x$ is not an eigenvalue of $M$, but satisfies $m_{M[T-v_0]}(x)=m_{M}(x)+1$. For $p\geq 2$, we get $L(M,x)=1$. The case $b\in A_1$ is analogous, with the least eigenvalue being replaced by the greatest eigenvalue.

If $x=a\in A_2$, then when we apply
\texttt{Diagonalize}$(M(T),-x)$, all vertices $v$ except $v_0$ are assigned a positive value. When processing $v_0$, the algorithm produces
\begin{equation}\label{eq:3}
    d_{v_{0}}=d^{(T_0)}_{v_0} - \sum_{i=1}^{p}\frac{m_{v_{0}v_{i}}^{2}}{d_{v_{i}}}.
\end{equation}
Since $L(M[T_0],x)=0$ by hypothesis, we have $d^{(T_0)}_{v_0}=0$. So the expression in~\eqref{eq:3} is negative. Theorem~\ref{inertia} implies that $|V(T)|-1$ eigenvalues of $M$ are greater than $x$ and one eigenvalue is less than $x$.

To prove part (a), summing the multiplicities, we obtain
\begin{eqnarray*}
m_M(a)&+& m_M(b) + \sum_{\lambda\in A_1\cap A_2} m_M(\lambda) \\&=& (p-1)^{\delta_{a1}} +(p-1)^{\delta_{b1}} + \sum_{\lambda\in A_1\cap A_2} \left( m_{M_2(T_0)}(\lambda)+\sum_{i=1}^pm_{M_1(T_i)}(\lambda)\right)\\ 
&\stackrel{\eqref{eq:multiplicities1}}{=}& n-2.
\end{eqnarray*}
This means that there are only two eigenvalues in $\Spec(M)$, namely $\lambda_{\max}(M)$ and  $\lambda_{\min}(M)$, establishing (a). 

Finally, we prove (f) using an argument based on the trace of a matrix. (Recall that the trace $\tr(M)$ of a square matrix $M$ is the sum of its diagonal elements; equivalently, it is the sum of its eigenvalues.) 
By our conclusions in (b) and (e) above, $\tr(M) = \tr(M_2(T_0)) +\sum_{i=1}^p \tr(M_1(T_i))$, we have
\begin{equation}\label{eq:part_f}
\lambda_{\max}+\lambda_{\min}= a+b,
\end{equation}
as required.
\end{proof}

We are now ready to prove Theorem~\ref{main_th}.
\begin{proof}[Proof of Theorem~\ref{main_th}]
We proceed by induction on $k$.
For $k=1$, let $\alpha<\beta\in\mathbb{R}$ and consider the set of real numbers $C_1=\{2\alpha-\beta,\alpha,\beta,2\beta-\alpha\}=\{\lambda_0^{(1)}<\lambda_1^{(1)}<\lambda_2^{(1)}<\lambda_3^{(1)}\}$. Let $T \in \mathcal{T}^\ast$ be a tree of diameter $d\in \{1,2\}=\{2k-1,2k\}$ with main root $v_0$. 

We define $M_1^{(1)}(T):=M_1^{(1)}$ as follows: set all diagonal values of $M_1^{(1)}$ as $\alpha$. By Lemma~\ref{lema_define_max}(ii), we may assign weights to the edges between $v_0$ and its children such that $\beta$ is the maximum eigenvalue of $M_1^{(1)}$.

Applying \texttt{Diagonalize}$(M^{(1)}_1,-\alpha)$, it is easy to see that $m_{M^{(1)}_1}(\alpha)=|V(T)|-2$ (in particular, this number is $0$ if $T$ has diameter $1$) and that $m_{M_1^{(1)}[T-v_0]}(\alpha)=m_{M_1^{(1)}}(\alpha)+1$. As a consequence, $L(M_1^{(1)},\alpha)=1$ if $\alpha$ is an eigenvalue of $M$. Moreover, by Theorem~\ref{inertia}, the two remaining eigenvalues must be $\lambda_{\min}(M_1^{(1)})$ and $\lambda_{\max}(M_1^{(1)})=\beta$. Considering the trace of $M_1^{(1)}$, we obtain $$\lambda_{\min}+(|V(T)|-2)\alpha + \beta= |V(T)|\cdot \alpha.$$
This shows that $\lambda_{\min}=2\alpha-\beta$, so that $\DSpec(M^{(1)}_1)\subseteq\{\lambda_0^{(1)},\lambda_1^{(1)},\lambda_2^{(1)}\}$. By our proof of Theorem~\ref{thm:simpleroots}, we know that $L(M_1^{(1)},\lambda_0)=L(M_1^{(1)},\lambda_2)=0$. This shows that $\{\lambda_0^{(1)},\lambda_1^{(1)},\lambda_2^{(1)}\}$ is strongly realizable for trees of height 1 in $\mathcal{T}^\ast$.

Next define $M_2^{(1)}(T):=M_2^{(1)}$ as follows: set all diagonal values of $M_2^{(1)}$ as $\beta$ and, by Lemma~\ref{lema_define_max}(iii), define the weights of the edges between $v_0$ and its children such that $\alpha$ is the minimum eigenvalue of $M_2^{(1)}$.


Applying \texttt{Diagonalize}$(M^{(1)}_2,-\beta)$, we again see that $m_{M^{(1)}_2}(\beta)=|V(T)|-2$, that $m_{M_2^{(1)}[T-v_0]}(\beta)=m_{M_2^{(1)}}(\beta)+1$ and that the remaining two eigenvalues are $\lambda_{\min}(M^{(1)}_2)=\alpha$ and $\lambda_{\max}(M^{(1)}_2)$. Considering the trace of $M^{(1)}_2$, we obtain $\lambda_{\max}(M^{(1)}_2)=2\beta-\alpha$, so that $\DSpec(M^{(1)}_2)\subseteq\{\lambda_1^{(1)},\lambda_2^{(1)},\lambda_3^{(1)}\}$. Here $L(M_2^{(1)},\lambda_1^{(1)})=L(M_2^{(1)},\lambda_3^{(1)})=0$. As a consequence, $\{\lambda_1^{(1)},\lambda_2^{(1)},\lambda_3^{(1)}\}$ is strongly realizable for trees of height 1 in $\mathcal{T}^\ast$.

So far, we have shown that items (i)-(iv) hold for the base of induction.

To prove (v), let $y_1=\beta-\alpha=\frac{\beta-\alpha}{2^{1-1}}$. Fix $\theta$ such that $0<\theta<y_1$. Observe that this interval is not empty, since $\beta>\alpha$. We define $M_{1,\theta}^{(1)}\in \mathcal{S}(T)$ as follows: the diagonal entries of $M_{1,\theta}^{(1)}$ are the same as $M_1^{(1)}$, except for the entry corresponding to $v_0$, which is $\alpha+\theta$. To ensure that the greatest eigenvalue of $M_{1,\theta}^{(1)}$ is equal to $\beta$, the weight $\omega$ assigned to the edges between $v_0$ and its children is defined by the solution of the following equation obtained by applying \texttt{Diagonalize}$(M_{1,\theta}^{(1)},-\beta)$ with root $v_0$:
\begin{equation}
\label{eq:6}
0 = (\alpha+\theta-\beta) - \sum_{w\neq v_0}\frac{\omega^2}{\alpha-\beta}.
\end{equation}
Note that $- \sum_{w\neq v_0}\frac{\omega^2}{\alpha-\beta}$ is positive, so (\ref{eq:6}) has a real solution $\omega$ if, and only if, $\alpha+\theta-\beta<0$, which is true since $\theta<\beta-\alpha$. As in the previous case, $\alpha$ has multiplicity $|V(T)|-2$ and $m_{M_{1,\theta}^{(1)}[T-v_0]}(\alpha)=m_{M_{1,\theta}^{(1)}}(\alpha)+1$. So far, we have $\DSpec(M_{1,\theta}^{(1)})\subseteq\{\lambda_{\min}(M_{1,\theta}^{(1)}),\lambda_1^{(1)},\lambda_2^{(1)}\}$. Finally, note that
$$\lambda_0^{(1)}+(|V|-2)\alpha+\beta = \tr(M_1^{(1)}) = \tr(M_{1,\theta}^{(1)})-\theta = \lambda_{\min}(M_{1,\theta}^{(1)})+(|V|-2)\alpha+\beta-\theta,$$
from which we obtain $\lambda_{\min}(M_{1,\theta}^{(1)})= \lambda_0^{(1)}+\theta$. As in the previous case, $L(M_{1,\theta}^{(1)},\lambda_2^{(1)})=0$ and $L(M_{1,\theta}^{(1)},\lambda_{0}^{(1)}+\theta)=0$.


To prove (vi), fix $\delta$ such that $\beta-\alpha>\delta>0>\alpha-\beta$. We define $M_{2,\delta}^{(1)}\in \mathcal{S}(T)$ as follows: the diagonal entries of $M_{2,\delta}^{(1)}$ are the same of $M_2^{(1)}$, except for the entry corresponding to $v_0$ which is $\beta+\delta$. To ensure that $\alpha$ is the least eigenvalue of $M_{2,\delta}^{(1)}$, the weight $\omega$ assigned to the edges between $v_0$ and its children is defined by the solution of the following equation obtained by applying \texttt{Diagonalize}$(M_{2,\delta}^{(1)},-\alpha)$ with root $v_0$:
\begin{equation}
\label{eq:7}
0 = (\beta+\delta-\alpha) - \sum_{w\neq v_0}\frac{\omega^2}{\beta-\alpha}.
\end{equation}
Note that $- \sum_{w\neq v_0}\frac{\omega^2}{\beta-\alpha}$ is negative, so (\ref{eq:7}) has a real solution $\omega$ if, and only if, $\beta+\delta-\alpha>0$, which is true since $\delta>\alpha-\beta$. As in the case of $M^{(1)}_2$, $\lambda_2^{(1)}=\beta$ has multiplicity $|V(T)|-2$ and $m_{M_{2,\delta}^{(1)}[T-v_0]}(\beta)=m_{M_{2,\delta}^{(1)}}(\beta)+1$. So far, we have            $\DSpec(M_{2,\delta}^{(1)})\subseteq \{\lambda_1^{(1)},\lambda_2^{(1)},\lambda_{\max}(M_{2,\delta}^{(1)})\}$. Finally, note that $$\alpha+(|V|-2)\beta+\lambda_{3}^{(1)}= \tr(M_2^{(1)}) = \tr(M_{2,\delta}^{(1)})-\delta = \alpha+(|V|-2)\beta+\lambda_{\max}(M_{2,\delta}^{(1)})-\delta,$$
from which we obtain $\lambda_{\max}(M_{2,\delta}^{(1)}) = \lambda_3^{(1)}+\delta$. As in the previous case, $L(M_{2,\delta}^{(1)},\lambda_1^{(1)})=0$ and $L(M_{2,\delta}^{(1)},\lambda_{3}^{(1)}+\delta)=0$.

Now, suppose by induction that for some $k\in\mathbb{N}$ we have a set $C_k=\{\lambda_0^{(k)}<\lambda_1^{(k)}<\cdots<\lambda_{2k+1}^{(k)}\}$ such that, for every $T'\in \mathcal{T}^\ast$ with height $k$ and diameter $d\in \{2k-1,2k\}$, there exist $M_1^{(k)}=M_1^{(k)}(T'),M_2^{(k)}=M_2^{(k)}(T')\in \mathcal{S}(T')$ satisfying the following properties:
\begin{enumerate}
    \item[(i)] $\DSpec(M_1^{(k)}) = \{\lambda_0^{(k)},\ldots,\lambda_{2k}^{(k)}\}$, if $d=2k$; $\DSpec(M_1^{(k)}) = \{\lambda_0^{(k)},\ldots,\lambda_{2k}^{(k)}\}\setminus\{\lambda_1^{(k)}\}$, if $d=2k-1$;
    \item[(ii)]  $\DSpec(M_2^{(k)})=\{\lambda_1^{(k)},\ldots,\lambda_{2k+1}^{(k)}\}$, if $d=2k$; $\DSpec(M_2^{(k)})=\{\lambda_1^{(k)},\ldots,\lambda_{2k+1}^{(k)}\}\setminus\{\lambda_{2k}^{(k)}\}$, if $d=2k-1$;
    \item[(iii)] $L(M_1^{(k)},\lambda_{2i}^{(k)})=0=L(M_2^{(k)},\lambda_{2i+1}^{(k)})$, for $i\in\{0,\ldots,k\}$;
    \item[(iv)] $m_{M_1^{(k)}[T'-v]}(\lambda_{2i-1}^{(k)})=m_{M_1^{(k)}}(\lambda_{2i-1}^{(k)})+1$ and $m_{M_2^{(k)}[T'-v]}(\lambda_{2i}^{(k)})=m_{M_2^{(k)}}(\lambda_{2i}^{(k)})+1$, for $i\in\{1,\ldots,k\}$;
    \end{enumerate}
    Moreover, the following hold:
    \begin{enumerate}
    \item[(v)] Let $y_k=\frac{\beta-\alpha}{2^{k-1}}$. For all $\theta\in(0,y_k)$, there exists $M_{1,\theta}^{(k)}$ such that  $$\DSpec(M_{1,\theta}^{(k)})\subseteq\{\lambda_0^{(k)}+\theta,\lambda_1^{(k)},\ldots,\lambda_{2k}^{(k)}\},$$ $L(M_{1,\theta}^{(k)},\lambda_{0}^{(k)}+\theta)=0$, and, for all $i\in\{1,\ldots,k\}$, we have $L(M_{1,\theta}^{(k)},\lambda_{2i}^{(k)})=0$ and $m_{M_{1,\theta}^{(k)}[T'-v]}(\lambda_{2i-1}^{(k)})=m_{M_{1,\theta}^{(k)}}(\lambda_{2i-1}^{(k)})+1$.
    \item[(vi)] For all $\delta\in(0,y_k)$, there exists $M_{2,\delta}^{(k)}$ such that $$\DSpec(M_{2,\delta}^{(k)})\subseteq\{\lambda_1^{(k)},\ldots,\lambda_{2k}^{(k)},\lambda_{2k+1}^{(k)}+\delta\},$$ $L(M_{2,\delta}^{(k)},\lambda_{2k+1}^{(k)}+\delta)=0$, and, for all $i\in\{1,\ldots,k\}$, we have $L(M_{2,\delta}^{(k)},\lambda_{2i-1}^{(k)})=0$ and $m_{M_{2,\delta}^{(k)}[T'-v]}(\lambda_{2i}^{(k)})=m_{M_{2,\delta}^{(k)}}(\lambda_{2i}^{(k)})+1$.
\end{enumerate}

Fix $\delta_k=\frac{\beta-\alpha}{2^k}\in(0,y_k)$ and $\theta_k=\frac{\beta-\alpha}{2^k}\in(0,y_k)$. Consider the set 
\begin{eqnarray}\label{def_set}
C_{k+1}&=&\{\lambda_0^{(k+1)}<\lambda_1^{(k+1)}<\lambda_2^{(k+1)}<\cdots<\lambda_{2k+1}^{(k+1)}<\lambda_{2k+2}^{(k+1)}<\lambda^{(k+1)}_{2k+3}\}\nonumber\\
&=&\{\lambda_0^{(k)}-\delta_k<\lambda_0^{(k)}<\lambda_1^{(k)}<\cdots<\lambda_{2k}^{(k)}<\lambda_{2k+1}^{(k)}+\delta_k<\lambda_{2k+1}^{(k)}+\delta_k+\theta_k\}
\end{eqnarray}
of cardinality
$2k+4$. We show that $C_{k+1}$ satisfies the required properties.

Let $T \in \mathcal{T}^\ast$ (rooted at a main root) with height $k+1$ and diameter $d\in \{2k+1,2k+2\}$. This means that $T = T_0 \odot (T_1,\ldots,T_p)$, where $p\geq1$, and that each $T_i \in \mathcal{T}^\ast$ has height $k$ and main root $v_i$, for all $i\in\{0,\ldots,p\}$. Recall that $T$ has diameter $2k+1$ if and only if $p=1$.

First, we define a matrix $M_1^{(k+1)}=M_1^{(k+1)}(T)$ with the structure of Figure~\ref{lemma_3.4fig}, where $M_1^{(k+1)}[T_0]=M_2^{(k)}(T_0)$ and $M_1^{(k+1)}[T_i]=M_1^{(k)}(T_i)$ for all $i\in\{1,\ldots,p\}$ are defined using the induction hypothesis. By parts (i) and (ii) of the induction hypothesis and Lemma~\ref{lema_define_max}, we can define the weights on the edges $v_0v_i$ so that $\lambda_{2k+2}^{(k+1)}=\lambda_{2k+1}^{(k)}+\delta_k$ is the maximum eigenvalue of $M_1^{(k+1)}$.

We wish to apply Lemma~\ref{multiplicidades}. By parts (i) to (iv) of the induction hypothesis, the hypotheses of the lemma are satisfied for $A_1=\{\lambda_0^{(k)},\ldots,\lambda_{2k}^{(k)}\}$, $A_2=\{\lambda_1^{(k)},\ldots,\lambda_{2k+1}^{(k)}\}$, $\mathcal{C}_1=\mathcal{C}_2=\mathcal{T}^\ast$, $\Lambda_1=\{\lambda_2^{(k)},\lambda_4^{(k)},\ldots,\lambda_{2k}^{(k)}\}$ and $\Lambda_2=\{\lambda_1^{(k)},\lambda_3^{(k)},\ldots,\lambda_{2k-1}^{(k)}\}$. Observe that $a=\lambda_0^{(k)}\in A_1$ and $b=\lambda_{2k+1}^{(k)} \in A_2$.

Given our choice of maximum eigenvalue, Lemma \ref{multiplicidades}(a)immediately implies that 
\begin{eqnarray}\label{specM1}
\DSpec(M_1^{(k+1)})&=&
\begin{cases}
\{\lambda_{\min},\lambda_0^{(k)},\lambda_1^{(k)},\ldots,\lambda_{2k}^{(k)},\lambda_{2k+1}^{(k)}+\delta_{k}\}, &\text{ if } d=2k+2,\\  \{\lambda_{\min},\lambda_1^{(k)},\ldots,\lambda_{2k}^{(k)},\lambda_{2k+1}^{(k)}+\delta_{k}\} ,&\text{ if } d=2k+1,
\end{cases}\\
&\subseteq& \{\lambda_{\min},\lambda_1^{(k+1)},\ldots,\lambda_{2k+2}^{(k+1)}\}.  \nonumber
\end{eqnarray}
Moreover, $\lambda_{\min} = \lambda_0^{(k)}-\delta_{k}=\lambda_0^{(k+1)}$ by Lemma~\ref{multiplicidades}(f). Therefore (i) is satisfied for $M_1^{(k+1)}$. We also obtain (iii) and (iv) by Lemma \ref{multiplicidades}.

Next, define $M_2^{(k+1)}=M_2^{(k+1)}(T)$ with the structure of Figure~\ref{lemma_3.4fig}, where $M_2^{(k+1)}[T_0]=M_{1,\theta_k}^{(k)}(T_0)$ and 
$M_2^{(k+1)}[T_i]=M_{2,\delta_k}^{(k)}(T_i)$ are defined based on the induction hypothesis.
By Lemma~\ref{lema_define_max}, we can define the weights on the edges $v_0v_i$ so that $\lambda^{(k+1)}_1=\lambda^{(k)}_{0}$ is the minimum eigenvalue of $M_2^{(k+1)}$. The induction hypothesis ensures that the hypotheses of Lemma~\ref{multiplicidades} are satisfied for $A_1=\{\lambda_1^{(k)},\ldots,\lambda_{2k}^{(k)},\lambda_{2k+1}^{(k)}+\delta_k\}$, $A_2=\{\lambda_0^{(k)}+\theta_k,\lambda_1^{(k)},\ldots,\lambda_{2k}^{(k)}\}$, $\mathcal{C}_1=\mathcal{C}_2=\mathcal{T}^\ast$, $\Lambda_1=\{\lambda_1^{(k)},\lambda_3^{(k)},\ldots,\lambda_{2k-1}^{(k)}\}$ and $\Lambda_2=\{\lambda_2^{(k)},\lambda_4^{(k)},\ldots,\lambda_{2k}^{(k)}\}$. Observe that $a=\lambda_0^{(k)}+\theta_k\in A_2$ and $b=\lambda_{2k+1}^{(k)}+\delta_k \in A_1$.

Furthermore, Lemma~\ref{multiplicidades}(a) ensures that
\begin{eqnarray}\label{specM2}
\DSpec(M_2^{(k+1)}) &=&
\begin{cases}
\{\lambda_0^{(k)},\lambda_1^{(k)},\ldots,\lambda_{2k}^{(k)},\lambda_{2k+1}^{(k)}+\delta_{k},\lambda_{\max}\},&\text{ if } d=2k+2,\\  \{\lambda_0^{(k)},\lambda_1^{(k)},\ldots,\lambda_{2k}^{(k)},\lambda_{\max}\},&\text{ if } d=2k+1.
\end{cases}\\
&\subseteq& 
\{\lambda_1^{(k+1)},\lambda_2^{(k+1)},\ldots,\lambda_{2k+2}^{(k+1)},\lambda_{\max}\}\nonumber
\end{eqnarray}
We have $\lambda_{\max} = \lambda_{2k+1}+\delta_{k}+\theta_k=\lambda_{2k+3}^{(k+1)}$ by Lemma \ref{multiplicidades}(f), proving (ii) for $M_{2}^{(k+1)}$. Items (iii) and (iv) also hold by Lemma~\ref{multiplicidades}.

It remains to prove (v) and (vi). We start with (v). Let $y_{k+1}=\delta_k=\frac{\beta-\alpha}{2^k}$ and let $\theta\in(0,y_{k+1})$.
Notice that, since $0<\theta<\delta_k<y_k$, item (vi) of the induction hypothesis applies to $M_{2,\theta}^{(k)}(T_0)$.

We define a matrix $M_{1,\theta}^{(k+1)}=M_{1,\theta}^{(k+1)}(T)$ with the structure of Figure~\ref{lemma_3.4fig}, where the induction hypothesis gives us $M_{1,\theta}^{(k+1)}[T_0]=M_{2,\theta}^{(k)}(T_0)$, $M_{1,\theta}^{(k+1)}[T_i]=M_1^{(k)}(T_i)$ for all $i\in\{1,\ldots,p\}$. By Lemma~\ref{lema_define_max} we can define the weights on the edges $v_0v_i$ such that $\lambda_{2k+2}^{(k+1)}=\lambda_{2k+1}^{(k)}+\delta_k$ is the maximum eigenvalue of $M_{1,\theta}^{(k+1)}$, since $\lambda_{2k+1}^{(k)}+\delta_k>\lambda_{2k+1}^{(k)}+\theta$. We again apply Lemma \ref{multiplicidades}, this time for $A_1=\{\lambda_0^{(k)},\ldots,\lambda_{2k}^{(k)}\}$, $A_2=\{\lambda_1^{(k)},\ldots,\lambda_{2k}^{(k)},\lambda_{2k+1}^{(k)}+\theta\}$, $\mathcal{C}_1=\mathcal{C}_2=\mathcal{T}^\ast$, $\Lambda_1=\{\lambda_2^{(k)},\lambda_4^{(k)},\ldots,\lambda_{2k}^{(k)}\}$ and $\Lambda_2=\{\lambda_1^{(k)},\lambda_3^{(k)},\ldots,\lambda_{2k-1}^{(k)}\}$. Observe that $a=\lambda_0^{(k)}\in A_1$ and $b=\lambda_{2k+1}^{(k)}+\theta \in A_2$.

Part (f) of Lemma~\ref{multiplicidades} implies that $\lambda_{\min}=\lambda_0^{(k)}-\delta_{k}+\theta$. 
Part (a) gives
\begin{eqnarray}\label{specMtheta}
\DSpec(M_{1,\theta}^{(k+1)}) &=&
\begin{cases}\{\lambda_0^{(k)}-\delta_{k}+\theta,\lambda_0^{(k)},\ldots,\lambda_{2k}^{(k)},\lambda_{2k+1}^{(k)}+\delta_{k}\}, &\text{ if } d=2k+2, \\
 \{\lambda_0^{(k)}-\delta_{k}+\theta,\lambda_1^{(k)},\ldots,\lambda_{2k}^{(k)},\lambda_{2k+1}+\delta_{k}^{(k)}\}, &\text{ if } d=2k+1. 
\end{cases} \nonumber\\
&\subseteq& \{\lambda_0^{(k+1)}+\theta,\lambda_1^{(k+1)},\ldots,\lambda_{2k+2}^{(k+1)}\}. 
\end{eqnarray}
The other properties of part (v) also follow from Lemma \ref{multiplicidades}.

For (vi), let $z_{k+1}=y_k-\theta_k=\frac{\beta-\alpha}{2^{k-1}}-\frac{\beta-\alpha}{2^k}=\frac{\beta-\alpha}{2^k}$ and fix $\delta\in(0,z_{k+1})$.
This gives $0<\delta<y_{k+1}\leq y_k-\theta_k$, so that $\delta+\theta_k<y_{k}$ and items (v) and (vi) of the induction hypothesis apply to $M_{1,\theta_k+\delta}^{(k)}(T_0)$ and $M_{2,\delta_k}^{(k)}(T_i)$.

Let $M_{2,\delta}^{(k+1)}=M_{2,\delta}^{(k+1)}(T)$ with the structure of Figure~\ref{lemma_3.4fig}, where we use the induction hypothesis to define $M_{2,\delta}^{(k+1)}[T_0]=M_{1,\theta_k+\delta}^{(k)}(T_0)$ and
$M_{2,\delta}^{(k+1)}[T_i]=M_{2,\delta_k}^{(k)}(T_i)$.
By Lemma~\ref{lema_define_max} we can define the weights on the edges $v_0v_i$ such that $\lambda_1^{(k+1)}=\lambda_{0}^{(k)}$ is the minimum eigenvalue of $M_{2,\delta}^{(k+1)}$.
We apply Lemma \ref{multiplicidades} once more, for $A_1=\{\lambda_1^{(k)},\ldots,\lambda_{2k}^{(k)},\lambda_{2k+1}^{(k)}+\delta_k\}$, $A_2=\{\lambda_0^{(k)}+\theta_k+\delta,\lambda_1^{(k)},\ldots,\lambda_{2k}^{(k)}\}$, $\mathcal{C}_1=\mathcal{C}_2=\mathcal{T}^\ast$, $\Lambda_1=\{\lambda_1^{(k)},\lambda_3^{(k)},\ldots,\lambda_{2k-1}^{(k)}\}$ and $\Lambda_2=\{\lambda_2^{(k)},\lambda_4^{(k)},\ldots,\lambda_{2k}^{(k)}\}$. Observe that $a=\lambda_0^{(k)}+\theta_k+\delta\in A_2$ and $b=\lambda_{2k+1}^{(k)}+\delta_k \in A_1$.

Given our choice of $\lambda_{\min}$, we have $\lambda_{\max} = \lambda_{2k+1}^{(k)}+\delta_{k}+\theta_k+\delta=\lambda_{2k+3}^{(k+1)}+\delta$ by Lemma \ref{multiplicidades}(f). Lemma \ref{multiplicidades}(a) gives
\begin{eqnarray}\label{specMdelta}
\DSpec(M_{2,\delta}^{(k+1)}) &=&
\begin{cases}\{\lambda_0^{(k)},\lambda_1^{(k)},\ldots,\lambda_{2k}^{(k)},\lambda_{2k+1}^{(k)}+\delta_{k},\lambda_{2k+1}^{(k)}+\delta_{k}+\theta_k+\delta\},&\text{ if } d=2k+2,\\
\{\lambda_0^{(k)},\lambda_1^{(k)},\ldots,\lambda_{2k}^{(k)},\lambda_{2k+1}^{(k)}+\delta_{k}+\theta_k+\delta\},&\text{ if } d=2k+1. 
\end{cases}\nonumber\\
&\subseteq& \{\lambda_1^{(k+1)},\ldots,\lambda_{2k+2}^{(k+1)},\lambda_{2k+3}^{(k+1)}+\delta\}.
\end{eqnarray}
The other properties of part (vi) also follow from Lemma \ref{multiplicidades}.

This concludes the step of induction, establishing Theorem~\ref{main_th}.
\end{proof}

\begin{remark}\label{relation_ck}
Note that the proof of Theorem~\ref{main_th} shows how the sets $C_k$ and $C_{k+1}$ relate to each other. Indeed, if $C_{k}=\{\lambda_{0},\ldots,\lambda_{2k+1}\}$ then $C_{k+1}=\{\lambda_{0}-\delta_{k},\lambda_{0},\ldots,\lambda_{2k},\lambda_{2k+1}+\delta_{k},\lambda_{2k+1}+\delta_{k}+\theta_{k}\}$
\end{remark}

\section{Proof of Theorem~\ref{thm:main} for other seeds}\label{sec:other_seeds}

To conclude the proof of Theorem~\ref{thm:main}, we prove it for unfoldings of the seeds $S'_d$ and $S''_d$. 

\begin{proof}[Proof of Theorem~\ref{thm:main}]
Let $T$ be an unfolding of $S_d'$ or $S_d''$ for some $d\geq 4$. Assume that $d\in\{2k+2,2k+3\}$ for some $k\geq 1$. Given arbitrary $\alpha<\beta$, we apply Theorem~\ref{main_th} (see also Remark~\ref{relation_ck}) to obtain sets $C_{k-1}=\{\lambda_{0},\ldots,\lambda_{2k-1}\}$ and $C_{k}=\{\lambda_{0}-\delta_{k-1},\lambda_{0},\ldots,\lambda_{2k-2},\lambda_{2k-1}+\delta_{k-1},\lambda_{2k-1}+\delta_{k-1}+\theta_{k-1}\}$ that satisfy conditions (i)-(vi) for trees $T^\ast \in \mathcal{T}^\ast$ of height $k-1$ and $k$, respectively.

In our construction, we consider each of the three possibilities for seeds $S_d'$ and $S_d''$ in Definition~\ref{def_seeds}.

\vspace{5pt}

\noindent \textbf{Case 1:} If $d=2k+2$ for some $k\geq 1$ and $T$ is an unfolding of $S_d'$, by Proposition~\ref{equivalence_other_seeds}(i), there exist $T_0\in \mathcal{T}^\ast$ of height $k-1$ and $T_1,\ldots,T_p\in \mathcal{T}^\ast$ of height $k$, for some $p\geq 2$, such that $T=T_0\odot(T_1,\ldots,T_p)$. 

We define a matrix $M\in \mathcal{S}(T)$ as follows: $M[T_0]=M_1^{(k-1)}(T_0)$ and $M[T_i]=M_1^{(k)}(T_i)$, for $i\in\{1,\ldots,p\}$, where $M_1$ denotes a matrix that satisfies (i) in Theorem~\ref{main_th}. To compute the spectrum of $M$ we apply Lemma~\ref{multiplicidades}. To this end, let $\mathcal{C}_1=\mathcal{C}_2=\mathcal{T}^\ast$ (where each tree is rooted at a main root). Let $k_1=k-1$, $k_2=k$, and consider $A_1=C_{k}\setminus\{\lambda_{2k-1}+\delta_{k-1}+\theta_{k-1}\}$ and $A_2=C_{k-1}\setminus\{\lambda_{2k-1}\}$. Note that $A_1 \cap A_2=\{\lambda_0,\ldots,\lambda_{2k-2}\}$ and that $(A_1 \cup A_2) \setminus (A_1 \cap A_2)=\{\lambda_0-\delta_{k-1},\lambda_{2k-1}+\delta_{k-1}\}$, so $a=\lambda_0-\delta_{k-1}$, $  b=\lambda_{2k-1}+\delta_{k-1}$. Set $\Lambda_1=\{\lambda_1,\lambda_3,\ldots,\lambda_{2k-3}\}$ and $\Lambda_2=\{\lambda_0,\lambda_2,\ldots,\lambda_{2k-2}\}$. By Theorem~\ref{main_th}(i), $M_1^{(k-1)}(T_0)$ is a strong realization of $A_1$ and $M_1^{(k)}(T_i)$ is a strong realization of $A_2$ for each $i\geq 1$.
By Theorem~\ref{main_th}(iii) and (iv)\footnote{Note that the same elements of $A_1\cap A_2$ play different roles with respect to $M_1^{(k-1)}(T_0)$ and $M_1^{(k)}(T_i)$, as the eigenvalues with even index with respect to the first matrix have odd index with respect to the second, and vice-versa.}, the following hold:
\begin{itemize}
    \item[(I)] For all $\lambda\in \Lambda_1$, $L(M_1^{(k)}(T_i),\lambda)=0$ and $m_{M_1^{(k-1)}[T_0-v_0]}(\lambda)=m_{M_1^{(k-1)}}(\lambda)+1$.
    
    \item[(II)] For all $\lambda \in \Lambda_2$, $L(M_1^{(k-1)}(T_0),\lambda)=0$
    and $m_{M_1^{(k)}[T_i-v_i]}(\lambda)=m_{M_1^{(k)}(T_i)}(\lambda)+1$.
\end{itemize}
Having verified the hypotheses, we are now ready to apply Lemma~\ref{multiplicidades}. Since $p>1$ and $a,b \in A_1$, Lemma~\ref{multiplicidades}(a) tells us that there exist $\lambda_{\min},\lambda_{\max}\in\mathbb{R}$ such that 
\begin{equation}\label{eq:dspec}
\DSpec(M) = \{\lambda_{\min},\lambda_0-\delta_{k-1},\lambda_0,\ldots,\lambda_{2k-2},\lambda_{2k-1}+\delta_{k-1},\lambda_{\max}\}.
\end{equation}
In particular, $|\DSpec(M)|=2k+3=d+1$, so that $q(T)=d+1$ in this case.

\vspace{5pt}

\noindent \textbf{
Case 2:} If $d=2k+3$ and $T$ is an unfolding of $S'_d$, by Proposition~\ref{equivalence_other_seeds}(ii), there exist $T_1,\ldots,T_p,T'_1,\ldots,T'_q\in \mathcal{T}^\ast, p,q\geq 1$, of height $k$ and $T_0,T'_0\in \mathcal{T}^\ast$ of height $k-1$ such that $T= (T_0\odot(T_1,\ldots,T_p))\odot (T'_0\odot(T'_1,\ldots,T'_q))$.

We define the matrix $M\in \mathcal{S}(T)$ in two parts. For the part that is related to $\tilde{T}=T_0\odot(T_1,\ldots,T_p)$, set
$M[T_0]=M_1^{(k-1)}(T_0)$ and $M[T_i]=M_1^{(k)}(T_i)$ for $i\in\{1,\ldots,p\}$, where $M_1$ denotes a matrix that satisfies (i) in Theorem~\ref{main_th}. By Lemma~\ref{lema_define_max}, we define the weights on the edges $v_0v_i$ so that $\lambda_{2k-1}+\delta_{k-1}+\theta_{k-1}$ is the maximum eigenvalue of $M[\tilde{T}]$. Note that, for $M[\tilde{T}]$, the hypotheses of Lemma~\ref{multiplicidades} are satisfied for the same sets $A_1$, $A_2$, $\Lambda_1$, $\Lambda_2$ defined in case 1 (for the same reasons).
    Then, by Lemma~\ref{multiplicidades}, there exist $\tilde{\lambda}_{\min},\tilde{\lambda}_{\max}=\lambda_{2k-1}+\delta_{k-1}+\theta_{k-1}\in\mathbb{R}$ such that 
    $$\DSpec (M[\tilde{T}]) \subseteq \{\tilde{\lambda}_{\min},\lambda_0-\delta_{k-1},\lambda_0,\ldots,\lambda_{2k-2},\lambda_{2k-1}+\delta_{k-1},\lambda_{2k-1}+\delta_{k-1}+\theta_{k-1}\}.$$ 
Moreover, $\tilde{\lambda}_{\min},\lambda_0,\ldots,\lambda_{2k-2},\lambda_{2k-1}+\delta_{k-1}+\theta_{k-1}$ satisfy Lemma~\ref{multiplicidades}(c), while the values $ \lambda_0-\delta_{k-1},\lambda_1,\ldots,\lambda_{2k-3},\lambda_{2k-1}+\delta_{k-1}$ satisfy Lemma~\ref{multiplicidades}(d).

For the part that is related to $\tilde{T}'=T_0'\odot(T_1',\ldots,T_p')$, we set $M[T'_0]=M_{2,\delta_{k-1}}^{(k-1)}(T'_0)$ and $M[T'_i]=M_2^{(k)}(T'_i)$ for all $i\in\{1,\ldots,q\}$, where $M_2$ denotes the matrix that satisfies (ii) and $M_{2,\delta}$ denotes the matrix that satisfies (vi) in Theorem~\ref{main_th}. By Lemma~\ref{lema_define_max}, we may define the weights on the edges $\{v'_0,v'_i\}$ such that $\lambda_{0}-\delta_{k-1}$ is the minimum eigenvalue of $M[\tilde{T}']$.
To compute the spectrum of $M[\tilde{T}']$ we apply Lemma~\ref{multiplicidades}. To this end, let $\mathcal{C}_1=\mathcal{C}_2=\mathcal{T}^\ast$ (where each tree is rooted at a main root). Let $k_1=k-1$, $k_2=k$, and consider $A_1=C_{k}\setminus\{\lambda_{0}-\delta_{k-1}\}$ and $A_2=(C_{k-1}\cup\{\lambda_{2k-1}+\delta_{k-1}\})\setminus\{\lambda_{0}, \lambda_{2k-1}\}$.
Note that $A_1 \cap A_2=\{\lambda_{1},\ldots,\lambda_{2k-2},\lambda_{2k-1}+\delta_{k-1}\}$, that $(A_1 \cup A_2) \setminus (A_1 \cap A_2)=\{\lambda_0,\lambda_{2k-1}+\delta_{k-1}+\theta_{k-1}\}$, and hence $a=\lambda_0$, $b=\lambda_{2k-1}+\delta_{k-1}+\theta_{k-1}$. Set $\Lambda_1=\{\lambda_2,\lambda_4,\ldots,\lambda_{2k-2}\}$ and $\Lambda_2=\{\lambda_1,\lambda_3,\ldots,\lambda_{2k-1}+\delta_{k-1}\}$. By Theorem~\ref{main_th}(vi), $M_{2,\delta_{k-1}}^{(k-1)}(T'_0)$ is a strong realization of $A_2$ and $M_2^{(k)}(T'_i)$ is a strong realization of $A_1$ for any $i\geq 1$. By Theorem~\ref{main_th}(iii), (iv) and (vi), the following hold:
\begin{itemize}
        \item[(I)] For all $\lambda\in \Lambda_1$, $L(M_2^{(k)}(T'_i),\lambda)=0$ and $$m_{M_{2,\delta_{k-1}}^{(k-1)}[T'_0-v'_0]}(\lambda)=m_{M_{2,\delta_{k-1}}^{(k-1)}(T'_0)}(\lambda)+1.$$
        \item[(II)] For all $\lambda \in \Lambda_2$, $L(M_{2,\delta_{k-1}}^{(k-1)}(T'_0),\lambda)=0$
        and $m_{M_2^{(k)}[T'_i-v'_i]}(\lambda)=m_{M_2^{(k)}(T'_i)}(\lambda)+1$.
    \end{itemize}
    Having verified the hypotheses, we are now ready to apply Lemma~\ref{multiplicidades}.
    Lemma~\ref{multiplicidades}(a) tells us that there exist $\tilde{\lambda}'_{\min}=\lambda_{0}-\delta_{k-1},\tilde{\lambda}_{\max}'\in\mathbb{R}$ such that 
    \begin{equation}\label{eq:dspecb}
    \DSpec(M[\tilde{T}']) \subseteq \{\lambda_0-\delta_{k-1},\lambda_0,\ldots,\lambda_{2k-2},\lambda_{2k-1}+\delta_{k-1},\lambda_{2k-1}+\delta_{k-1}+\theta_{k-1},\tilde{\lambda}'_{\max}\}.
    \end{equation}
Moreover, $\lambda_0-\delta_{k-1},\lambda_1,\ldots,\lambda_{2k-1}+\delta_{k-1},\tilde{\lambda}_{\max}'$ satisfy Lemma~\ref{multiplicidades}(c), while the values $\lambda_0,\lambda_2,\ldots,\lambda_{2k-2},\lambda_{2k-1}+\delta_{k-1}+\theta_{k-1}$ satisfy Lemma~\ref{multiplicidades}(d).

To conclude the proof we apply Lemma~\ref{multiplicidades} to $\tilde{T}\odot \tilde{T}'$ using the matrices defined above. Here, $\mathcal{C}_1=\{\tilde{T}\}$,  $\mathcal{C}_2=\{\tilde{T}'\}$, $k_1=k_2=k+1$, $A_1=\{\tilde{\lambda}_{\min},\lambda_0-\delta_{k-1},\lambda_0,\ldots,\lambda_{2k-2},\lambda_{2k-1}+\delta_{k-1},\lambda_{2k-1}+\delta_{k-1}+\theta_{k-1}\}$, $A_2=\{\lambda_0-\delta_{k-1},\lambda_0,\ldots,\lambda_{2k-2},\lambda_{2k-1}+\delta_{k-1},\lambda_{2k-1}+\delta_{k-1}+\theta_{k-1},\tilde{\lambda}'_{\max}\}$, hence $a=\tilde{\lambda}_{\min}$, $b=\tilde{\lambda}'_{\max}$. Set $\Lambda_1=\{\lambda_0,\lambda_2,\ldots,\lambda_{2k-2},\lambda_{2k-1}+\delta_{k-1}+\theta_{k-1}\}$ and $\Lambda_2=\{\lambda_0-\delta_{k-1},\lambda_1,\lambda_3,\ldots,\lambda_{2k-3},\lambda_{2k-1}+\delta_{k-1}\}$. Since $a\in A_1$, $b\in A_2$ and $p=1$, it follows that there exist $\lambda_{\min}$ and $\lambda_{\max}$ such that
\begin{equation}
\DSpec(M) = \{\lambda_{\min},\lambda_0-\delta_{k-1},\lambda_0,\ldots,\lambda_{2k-2},\lambda_{2k-1}+\delta_{k-1},\lambda_{2k-1}+\delta_{k-1}+\theta_{k-1},\lambda_{\max}\}.
\end{equation}
In particular, $|\DSpec(M)|=2k+4=d+1$, so that $q(T)=d+1$ in this case.

\vspace{5pt}

\noindent \textbf{Case 3:} If $d=2k+3$ for some $k\geq 1$ and $T$ is an unfolding of $S''_d$, by Proposition~\ref{equivalence_other_seeds}(iii), there exist $T_0\in \mathcal{T}^\ast$ of height $k-1$ and $T_1,\ldots,T_p,T'_0,T'_1,\ldots,T'_q\in \mathcal{T}^\ast$ of height $k$, where $p,q\geq 1$, such that $T= (T_0\odot(T_1,\ldots,T_p))\odot (T'_0\odot(T'_1,\ldots,T'_q))$. 

We define the matrix $M\in \mathcal{S}(T)$ in two parts. For the part that is related to $\tilde{T}=T_0\odot(T_1,\ldots,T_p)$, set
$M[T_0]=M_1^{(k-1)}(T_0)$ and $M[T_i]=M_1^{(k)}(T_i)$ for $i\in\{1,\ldots,p\}$, where $M_1$ denotes a matrix that satisfy (i) in Theorem~\ref{main_th}. By Lemma~\ref{lema_define_max}, we may define the weights on the edges $\{v_0,v_i\}$ such that $\lambda_{2k-1}+\delta_{k-1}+\theta_{k-1}$ is the maximum eigenvalue of $M[T_0\odot(T_1,\ldots,T_p)]$. Note that, for $M[T_0\odot(T_1,\ldots,T_p)]$, all hypotheses of Lemma~\ref{multiplicidades} are satisfied for the same reasons described in case 1 above. Then, by Lemma~\ref{multiplicidades}, there exist $\tilde{\lambda}_{\min},\tilde{\lambda}_{\max}=\lambda_{2k-1}+\delta_{k-1}+\theta_{k-1}\in\mathbb{R}$ such that 
$$\DSpec (M[\tilde{T}]) \subseteq \{\tilde{\lambda}_{\min},\lambda_0-\delta_{k-1},\lambda_0,\ldots,\lambda_{2k-2},\lambda_{2k-1}+\delta_{k-1},\lambda_{2k-1}+\delta_{k-1}+\theta_{k-1}\}.$$ 
Moreover, $\tilde{\lambda}_{\min},\lambda_0,\ldots,\lambda_{2k-2},\lambda_{2k-1}+\delta_{k-1}+\theta_{k-1}$ satisfy Lemma~\ref{multiplicidades}(c), while $ \lambda_0-\delta_{k-1},\lambda_1,\ldots,\lambda_{2k-3},\lambda_{2k-1}+\delta_{k-1}$ satisfy Lemma~\ref{multiplicidades}(d).

For the part that is related to $\tilde{T}'=T_0'\odot(T_1',\ldots,T_p')$, set $M[T'_0]=M_{1,\theta_{k}}^{(k)}(T'_0)$, $M[T'_i]=M_2^{(k)}(T'_i)$, for $i\in\{1,\ldots,q\}$, where $M_2$ denotes a matrix that satisfies (ii) and $M_{1,\theta}$ denotes a matrix that satisfies (v) in Theorem~\ref{main_th}. By Lemma~\ref{lema_define_max}, we may define the weights on the edges $\{v'_0,v'_i\}$ such that $\lambda_{0}-\delta_{k-1}$ is the minimum eigenvalue of $M[\tilde{T}']$.
To compute the spectrum of $M[\tilde{T}']$ we apply Lemma~\ref{multiplicidades}. To this end, let $\mathcal{C}_1=\mathcal{C}_2=\mathcal{T}^\ast$, let $k_1=k_2=k$, and consider $A_1=C_{k}\setminus\{\lambda_{0}-\delta_{k}\}$ and $A_2=(C_{k}\cup\{\lambda_{0}-\delta_{k-1}+\theta_{k}\})\setminus\{\lambda_{0}-\delta_{k-1},\lambda_{2k-1}+\delta_{k-1}+\theta_{k-1}\}$.
    Note that $A_1 \cap A_2=\{\lambda_{0},\ldots,\lambda_{2k-2},\lambda_{2k-1}+\delta_{k-1}\}$, that $(A_1 \cup A_2) \setminus (A_1 \cap A_2)=\{\lambda_{0}-\delta_{k-1}+\theta_{k},\lambda_{2k-1}+\delta_{k-1}+\theta_{k-1}\}$, and hence $a=\lambda_{0}-\delta_{k-1}+\theta_{k}$, $b=\lambda_{2k-1}+\delta_{k-1}+\theta_{k-1}$. Set $\Lambda_1=\{\lambda_2,\lambda_4,\ldots,\lambda_{2k-2}\}$ and $\Lambda_2=\{\lambda_1,\lambda_3,\ldots,\lambda_{2k-1}+\delta_{k-1}\}$. 
    By Theorem~\ref{main_th}(vi), $M_{1,\theta_{k}}^{(k)}(T'_0)$ is a strong realization of $A_2$ and $M_2^{(k)}(T'_i)$ is a strong realization of $A_1$ for any $i\geq 1$.
    By Theorem~\ref{main_th}(iii), (iv) and (v), the following hold:
    \begin{itemize}
        \item[(I)] For all $\lambda\in \Lambda_1$, $L(M_2^{(k)}(T'_i),\lambda)=0$ and $m_{M_{1,\theta_{k}}^{(k)}[T'_0-v'_0]}(\lambda)=m_{M_{1,\theta_{k}}^{(k)}(T'_0)}(\lambda)+1$.
        \item[(II)] For all $\lambda \in \Lambda_2$, $L(M_{1,\theta_{k}}^{(k)}(T'_0),\lambda)=0$
        and $m_{M_2^{(k)}[T'_i-v'_i]}(\lambda)=m_{M_2^{(k)}(T'_i)}(\lambda)+1$.
    \end{itemize}
    Having verified the hypotheses, we are now ready to apply Lemma~\ref{multiplicidades}.
    Lemma~\ref{multiplicidades}(a) tells us that there exist $\tilde{\lambda}'_{\min}=\lambda_{0}-\delta_{k-1},\tilde{\lambda}'_{\max}\in\mathbb{R}$ such that 
    \begin{equation}\label{eq:dspecc}
    \DSpec(M[\tilde{T}']) \subseteq \{\lambda_0-\delta_{k-1},\lambda_0,\ldots,\lambda_{2k-2},\lambda_{2k-1}+\delta_{k-1},\lambda_{2k-1}+\delta_{k-1}+\theta_{k-1},\tilde{\lambda}'_{\max}\}.
    \end{equation}
   Moreover,  $\lambda_0-\delta_{k-1},\lambda_1,\ldots,\lambda_{2k-1}+\delta_{k-1},\tilde{\lambda}'_{\max}$ satisfy Lemma~\ref{multiplicidades}(c), while the values $\lambda_0,\lambda_2,\ldots,\lambda_{2k-2},\lambda_{2k-1}+\delta_{k-1}+\theta_{k-1}$ satisfy Lemma~\ref{multiplicidades}(d).

As in case 2, we conclude the proof by applying Lemma~\ref{multiplicidades} to $\tilde{T}\odot \tilde{T}'$. This gives $\lambda_{\min},\lambda_{\max}$ such that
\begin{equation}
\DSpec(M) = \{\lambda_{\min},\lambda_0-\delta_{k-1},\lambda_0,\ldots,\lambda_{2k-2},\lambda_{2k-1}+\delta_{k-1},\lambda_{2k-1}+\delta_{k-1}+\theta_{k-1},\lambda_{\max}\}.
\end{equation}
In particular, $|\DSpec(M)|=2k+4=d+1$, so that $q(T)=d+1$ in this case.
\end{proof}

We observe that our proof of Theorem~\ref{thm:main} using Theorem~\ref{main_th} allows us to ask more about the spectrum of a realization of a diminimal tree. For instance, we may require it to be integral. 
\begin{corollary}
Let $d$ be a positive integer. Let $\mathcal{T}(S_d)$, $\mathcal{T}(S'_d)$ and $\mathcal{T}(S''_d)$ be the families of trees of diameter $d$ generated by the seeds $S_d$, $S'_d$ and $S''_d$, respectively, where $S'_d$ is defined for $d\geq 4$ and $S''_d$ for odd values of $d \geq 5$. For every $T\in \mathcal{T}(S_d) \cup \mathcal{T}(S'_d) \cup \mathcal{T}(S''_d)$, there is a real symmetric matrix $M(T)$ whose spectrum is integral with underlying tree $T$ and $|\DSpec(T)|=d+1$.
\end{corollary}

\begin{proof}
The proof follows the same steps of the proof of Theorem~\ref{thm:main}. However, when $d \in \{2k,2k+1\}$ and we apply Theorem~\ref{main_th} to produce the set $C_k$, we start the proof by fixing an arbitrary integer $\alpha$ and by choosing $\beta=\alpha+2^{k-1}$, so that $\frac{\beta-\alpha}{2^{k-1}}$ is an integer. Then $C_1=\{2\alpha-\beta,\alpha,\beta,2\beta-\alpha\}$ is integral and the elements $\delta_j,\theta_j$ are integers for all $j\leq k-1$. Remark~\ref{relation_ck} ensures that the sets $C_j$ are integral for all $j\leq k$. This gives the desired conclusion for unfoldings of $S_d$.

For the other seeds, we need to go back to the proof of Theorem~\ref{thm:main}. For instance, assume that we are in the case $d=2k+2$ and we have an unfolding of $S'_{d}$. With the choices that we made for $S_{d}$, if we repeat the proof of Theorem~\ref{thm:main} until we get to~\eqref{eq:dspec}, we deduce that all elements of $\DSpec(M)$ are integers except possibly $\lambda_{\min}$ and $\lambda_{\max}$. However, when we applied Lemma~\ref{multiplicidades} to define $M(T)$, it was not necessary to assign weights to the edges joining the roots of the trees $T_0\odot (T_1,\ldots,T_p)$. By Lemma~\ref{lema_define_max}, we can assign these weights in a way that $\lambda_{\max}$ is equal to any value greater than $\lambda_{2k-1}+\delta_{k-1}$, and we may choose this value to be an integer. Moreover, Lemma~\ref{multiplicidades}(f) tells us that $\lambda_{\min}+\lambda_{\max}=a+b$, where $a$ and $b$ are both known to be integers. Therefore $\lambda_{\min}$ is also an integer and the result follows.  

Unfoldings of the other two seeds may be dealt with using similar arguments.
\end{proof}

\section{Example}\label{sec:example}
In this section we provide an example to illustrate that our proofs may be used to construct matrices associated with diminimal trees. In this example, we construct a symmetric matrix whose underlying graph is the seed $S_9$ (of diameter 9) with exactly 10 distinct eigenvalues. It is based on the matrix 
$M_{1}^{(5)}\in\mathcal{S}(S_{9})$ defined in Theorem~\ref{main_th}. We choose $\alpha=0$ and $\beta=32$, and after $k=5$ steps we obtain the matrix $M^{(5)}_1 (S_{9})\in\mathcal{S}(S_{9})$ with integral spectrum given by $$\Spec(M^{(5)}_1) = \{-62^{[1]}, -56^{[1]}, -48^{[2]}, -32^{[4]},   0^{[8]},  32^{[8]}\-,  80^{[4]}\-, 104^{[2]}, 116^{[1]}, 122^{[1]}\}.$$ It is depicted in Figure~\ref{S9}, where vertex weights denote the diagonal entries and edge weights denote the off-diagonal nonzero entries. 
In a git repository\footnote{\url{https://github.com/Lucassib/Diminimal-Graph-Algorithm} or \url{https://lucassib-diminimal-graph-algorithm-st-app-0t3qu7.streamlit.app/}}, readers can access an algorithm based on the proof of Theorem \ref{main_th} to compute a matrix $M\in\mathcal{S}(S_{d})$ where the input parameters are $\alpha$, $\beta$ and $d$, where $k=\ceil{\frac{d}{2}}$.

\begin{figure}[H]
\includegraphics[]{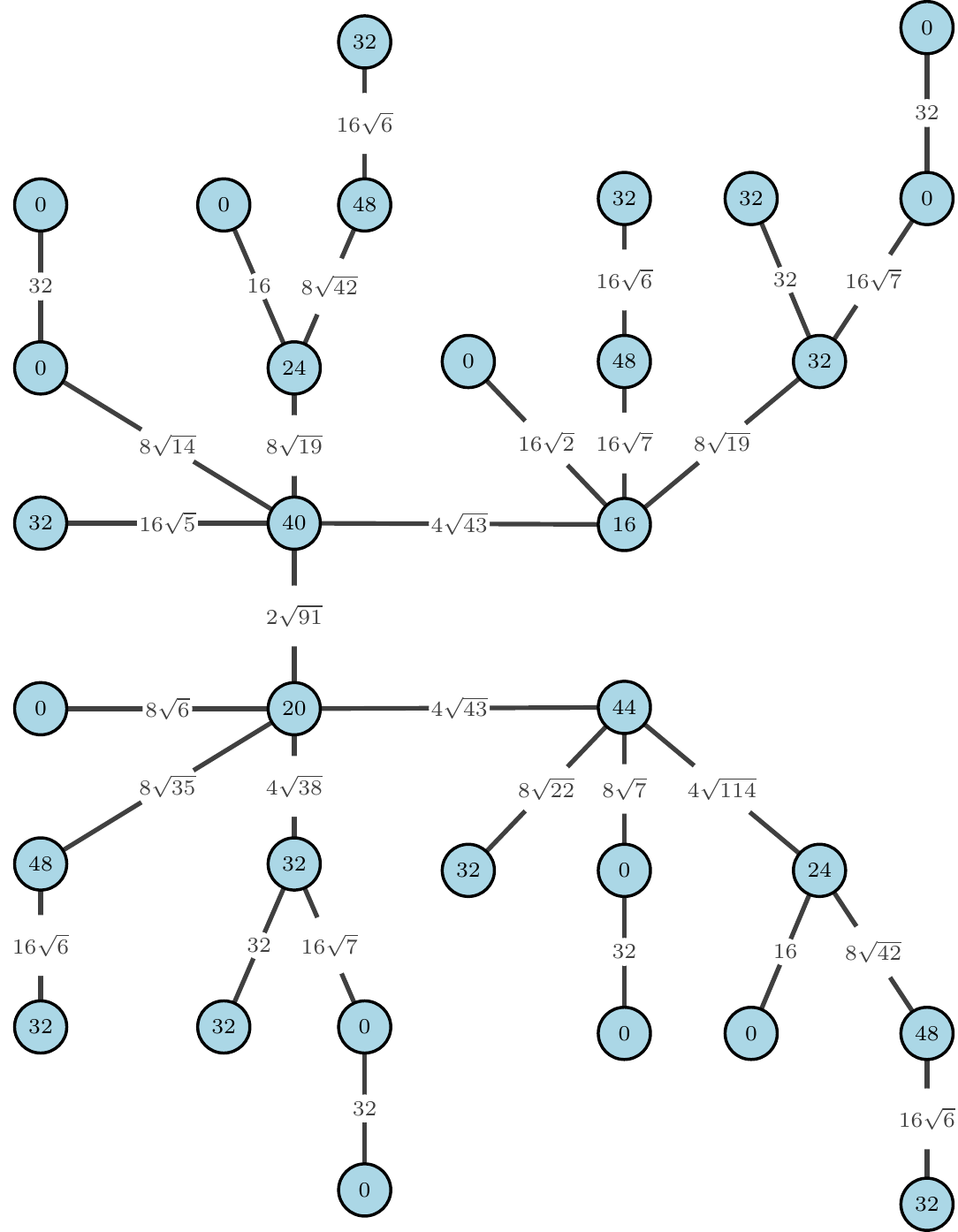}
\caption{\label{S9} Matrix $M^{(5)}_1 \in \mathcal{S}(S_{9})$.}
\end{figure}
\section*{Acknowledgments}
This work is partially supported by
MATH-AMSUD under project GSA, brazilian team financed by CAPES under project 88881.694479/2022-01. L. E. Allem acknowledges the support of FAPERGS 21/2551-
0002053-9. C.~Hoppen acknowledges the support of FAPERGS~19/2551-0001727-8  and CNPq (Proj.\ 315132/2021-3). V. Trevisan acknowledges partial support of CNPq grants 409746/2016-9 and 310827/2020-5, and FAPERGS grant PqG 17/2551-0001.
CNPq is the National Council for Scientific and Technological
Development of Brazil.

\bibliographystyle{amsplain}
\bibliography{mybibliography}

\appendix

\section{Additional results}

We illustrate how Proposition~\ref{equivalence_other_seeds} can be proved by providing a detailed proof of item (ii). The proofs of (i) and (iii) are analogous.  Proposition~\ref{equivalence_other_seeds}(ii) states that $T\in\mathcal{T}(S'_{2k+3})$ if, and only if, there exist $T_1,\ldots,T_p,T'_1,\ldots,T'_q\in \mathcal{T}^\ast, p,q\geq 1$ of height $k$ and $T_0,T'_0\in \mathcal{T}^\ast$ of height $k-1$ such that $$T=(T_0\odot(T_1,\ldots,T_p))\odot (T'_0\odot(T'_1,\ldots,T'_q)).$$

Let $T$ be a tree and $k\geq1$. The case $k=1$ ($S_5'=P_6$) is simple, so we concentrate in the case $k\geq  2$. First assume that there exist $T_1,\ldots,T_p,T'_1,\ldots,T'_q\in \mathcal{T}^\ast$ of height $k$, where $p,q\geq 1$, and $T_0,T'_0\in \mathcal{T}^\ast$ of height $k-1$ such that $T=(T_0\odot(T_1,\ldots,T_p))\odot (T'_0\odot(T'_1,\ldots,T'_q))$ (see Figure~\ref{fig:left}). 
Note that all paths of length $2k+3$ in $T$ may be decomposed as $Pv_0v_0'Q$ where $P$ is a path of length $k$ joining a leaf of some $T_i$ to its root $v_i$ and $Q$ is a path of length $k$ joining the root $v_j'$ of some $T_j'$ to one of its leaves. In particular, no such path uses vertices in $V(T_0-v_0)\cup V(T_0'-v_0')$ nor vertices in two different components of some $T_i-v_i$ or $T_j-v_j'$. 

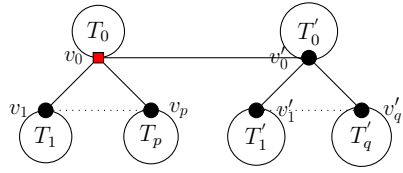
\begin{figure}
    \centering
    \begin{tikzpicture}[scale=.7,auto=left,every node/.style={circle,scale=0.5}]

\path(0,0)node[shape=circle,scale=1.5,draw] (1) {$T_{1}$}
     (0,.5)node[shape=circle,label=left:\Large{$v_{1}$},draw,fill=black] (11) {}
     (2,0)node[shape=circle,scale=1.5,draw] (2) {$T_{p}$}
     (2,.5)node[shape=circle,label=right:\Large{$v_{p}$},draw,fill=black] (22) {}
      (1,2)node[shape=circle,scale=1.5,draw] (0) {$T_{0}$}
     (1,1.5)node[shape=rectangle,label=left:\Large{$v_{0}$},draw,fill=red] (00) {}

     (4,0)node[shape=circle,scale=1.5,draw] (4) {$T_{1}^{'}$}
     (4,.5)node[shape=circle,label=right:\Large{$v_{1}'$},draw,fill=black] (44) {}
     (6,0)node[shape=circle,scale=1.5,draw] (5) {$T_{q}^{'}$}
     (6,.5)node[shape=circle,label=right:\Large{$v_{q}'$},draw,fill=black] (55) {}
      (5,2)node[shape=circle,scale=1.5,draw] (3) {$T_{0}^{'}$}
     (5,1.5)node[shape=circle,label=left:\Large{$v_{0}'$},draw,fill=black] (33) {};

     \draw[-](11)--(00);
     \draw[-](22)--(00);
     \draw[dotted](11)--(22);

     \draw[-](33)--(00);
     \draw[-](33)--(44);
     \draw[-](33)--(55);
     \draw[dotted](44)--(55);

\end{tikzpicture}
\caption{$T=(T_0\odot(T_1,\ldots,T_p))\odot (T'_0\odot(T'_1,\ldots,T'_q))$\label{fig:left}}
\end{figure}

By Proposition~\ref{prop_equivalence}, we know that the trees 
$T_1,\ldots,T_p,T'_1,\ldots,T'_q$ are unfoldings of $S_{2k-1}$ or $S_{2k}$, and that $T_0,T'_0$ are unfoldings of $S_{2k-3}$ or $S_{2k-2}$. Recall that, in part (ii) of the proof of Proposition~\ref{prop_equivalence}, given $j\geq 2$, we were able to fold the pair $(S_{2j-3},S_{2j-3})$ in $S_{2j}=S_{2j-3}\odot(S_{2j-3},S_{2j-3})$ to $S_{2j-3}\odot S_{2j-3}=S_{2j-1}$ without affecting the diameter of the tree. This \emph{does not} mean that $S_{2j}$ can always be folded onto $S_{2j-1}$, but instead that folding can be performed if the diameter of the tree is not modified.

For the tree $T$ in this proposition, where maximum paths have the structure mentioned above, this means that any $T_i$ or $T_i'$ with $i>0$ may be folded directly to $S_{2k-1}$ or may first be folded to $S_{2k}=S_{2k-3} \odot (S_{2k-3},S_{2k-3})$, which can in turn be folded to $S_{2k-3} \odot S_{2k-3}=S_{2k-1}$. Similarly, if $k\geq 3$, $T_0$ and $T_0'$ may be folded directly to $S_{2k-3}$ or may first be folded to $S_{2k-2}=S_{2k-5} \odot (S_{2k-5},S_{2k-5})$ and then to $S_{2k-5} \odot S_{2k-5}=S_{2k-3}$. For $k=2$, $T_0$ and $T_0'$ have height 1, so they are equal to $S_{1}$ or they are stars that can be folded into $S_1$.

Combining this, we conclude that $T$ can be folded to 
$$T'=(S_{2k-3}\odot(S_{2k-1},\ldots,S_{2k-1}))\odot (S_{2k-3}\odot(S_{2k-1},\ldots,S_{2k-1})),$$ with $p$ terms in the first vector and $q$ terms in the second. Now, if $p>1$ or $q>1$, we can fold each $(S_{2k-1},\ldots,S_{2k-1})$ onto a single $S_{2k-1}$, without decreasing the diameter. This results in $(S_{2k-3}\odot S_{2k-1})\odot (S_{2k-3}\odot S_{2k-1})=S'_{2k+3}$, as required. Figure~\ref{fig:case} illustrates this case.

\begin{figure}[H]
\centering
\begin{minipage}[h]{.5\textwidth}
\centering 
\begin{subfigure}{0.49\textwidth}
\centering 
            \begin{tikzpicture}[scale=.7,auto=left,every node/.style={circle,scale=0.5}]

\path(0,-0.2)node[shape=circle,scale=1.5,draw] (1) {$S_{2k-1}$}
     (0,.5)node[shape=circle,draw,fill=red] (11) {}
     (2,-.2)node[shape=circle,scale=1.5,draw] (2) {$S_{2k-1}$}
     (2,.5)node[shape=circle,draw,fill=red] (22) {}
      (1,2.2)node[shape=circle,scale=1.5,draw] (0) {$S_{2k-3}$}
     (1,1.5)node[shape=circle,draw,fill=red] (00) {}

     (4,-.2)node[shape=circle,scale=1.5,draw] (4) {$S_{2k-1}$}
     (4,.5)node[shape=circle,draw,fill=red] (44) {}
     (6,-.2)node[shape=circle,scale=1.5,draw] (5) {$S_{2k-1}$}
     (6,.5)node[shape=circle,draw,fill=red] (55) {}
      (5,2.2)node[shape=circle,scale=1.5,draw] (3) {$S_{2k-3}$}
     (5,1.5)node[shape=circle,draw,fill=red] (33) {};

     \draw[-](11)--(00);
     \draw[-](22)--(00);
     \draw[dotted](11)--(22);

     \draw[-](33)--(00);
     \draw[-](33)--(44);
     \draw[-](33)--(55);
     \draw[dotted](44)--(55);

\end{tikzpicture}
 \end{subfigure}
 
 \end{minipage}\hfill
\begin{minipage}[h]{.5\textwidth}
\centering

 \begin{subfigure}{0.49\textwidth}
     \centering
\begin{tikzpicture}[scale=.7,auto=left,every node/.style={circle,scale=0.5}]

\path(0,-0.2)node[shape=circle,scale=1.5,draw] (1) {$S_{2k-1}$}
     (0,.5)node[shape=circle,draw,fill=red] (11) {}
     (0,2.2)node[shape=circle,scale=1.5,draw] (0) {$S_{2k-3}$}
     (0,1.5)node[shape=circle,draw,fill=red] (00) {}

     (2,-.2)node[shape=circle,scale=1.5,draw] (4) {$S_{2k-1}$}
     (2,.5)node[shape=circle,draw,fill=red] (44) {}
    (2,2.2)node[shape=circle,scale=1.5,draw] (3) {$S_{2k-3}$}
     (2,1.5)node[shape=circle,draw,fill=red] (33) {};

     \draw[-](11)--(00);

     \draw[-](33)--(00);
     \draw[-](33)--(44);

\end{tikzpicture}
 \end{subfigure}
 
 \end{minipage}
\caption{Folding of $T'$ into $(S_{2k-3}\odot S_{2k-1})\odot (S_{2k-3}\odot S_{(2k-1})=S'_{2k+3}$.\label{fig:case}}
\end{figure}
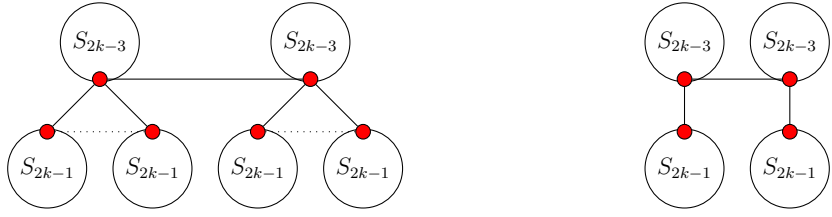

For the converse, our proof is by induction on the number of branch decompositions performed on the seed $S'_{2k+3}$ to produce $T$. If no CBD was performed, then $T=S'_{2k+3}=(S_{2k-3}\odot S_{2k-1})\odot (S_{2k-3}\odot S_{2k-1})$ and we have $T=(T_0\odot T_1) \odot (T_0'\odot T_1')$ for $T_0=T_0'=S_{2k-3}$ (of height $k-1$) and $T_1=T_1'=S_{2k-1}$ (of height $k$). 

Now suppose that if $T\in\mathcal{T}(S'_{2k+3})$ has been formed after a sequence of $\ell$ branch decompositions, then there exist $T_0,T_1,\ldots,T_p,T_0',T'_1,\ldots,T'_q\in \mathcal{T}^\ast$ as in the statement of the theorem for which 
\begin{equation}\label{eq:2}
T=(T_0\odot(T_1,\ldots,T_p))\odot (T'_0\odot(T'_1,\ldots,T'_q)).    
\end{equation} Note that the central edge of $T$ is $\{v_0,v_0'\}$ and the tree is rooted at $v_0$.

We claim that if we perform an additional $s$-CBD to $T$, we still obtain a decomposition as in (iii). Indeed, let $U$ be the tree obtained after performing an $s$-CBD of a branch $B$ at $v\in V(T)$. First assume that $v\notin \{v_0,v_0'\}$. Without loss of generality, assume that $v\in V(T_i)$, so that, in case $i=0$, $v$ is not the root of $T_0$. Since the diameter remains the same, $B$ must be entirely contained in $T_i$. By Proposition~\ref{prop_equivalence}, the tree $U_i$ obtained after performing an $s$-CBD of branch $B$ at $v\in V(T_i)$ lies in $\mathcal{T}^\ast$. In particular, if we replace $T_i$ by $U_i$ in~\eqref{eq:2}, we get the desired decomposition of $U$.

Next assume that $v =  v_0$ (the case $v =  v'_0$ is analogous). Let $B$ be the branch at $v_0$ involved in the duplication. This is not the branch that contains $v_0'$, otherwise the diameter would increase. If $B$ is entirely contained in $T_0$, we may repeat the above argument. Otherwise, $B=T_i$ for some $i$, and  
$$T=(T_0\odot(T_1,\ldots,T_i,T_i^{(1)},\ldots,T_i^{(s)},T_{i+1},\ldots,T_p))\odot (T'_0\odot(T'_1,\ldots,T'_q)),$$
where each $T_i^{(j)}$ is a copy of $T_i$. This concludes the proof.

\end{document}